\documentclass[a4paper,10pt]{article}

\usepackage[english]{babel}
\usepackage{amsmath}
\usepackage{amsfonts}
\usepackage[latin1]{inputenc}
\usepackage[T1]{fontenc}
\usepackage{amssymb}
\usepackage{color}
\usepackage{amsthm}
\usepackage{latexsym}
\usepackage{graphicx}
\usepackage{bbm}
\usepackage{url}
\usepackage[square, numbers, sort]{natbib}
\usepackage{todonotes}
\usepackage{colonequals}
\usepackage{hyperref}

\newtheorem{theorem}{Theorem}[section]
\newtheorem{definition}[theorem]{Definition}
\newtheorem{prop}[theorem]{Proposition}
\newtheorem{lemma}[theorem]{Lemma}
\newtheorem{assumption}[theorem]{Assumption}
\newtheorem{remark}[theorem]{Remark}
\newtheorem{coro}[theorem]{Corollary}
\newtheorem{example}[theorem]{Example}

\newcommand{\R}{{\mathbb{R}}}
\newcommand{\N}{{\mathbb{N}}}
\newcommand{\diam}{\operatorname{diam}} 
\newcommand{\divergence}{\operatorname{div}} 

\title{Unsaturated subsurface flow with surface water and nonlinear in- and outflow conditions\thanks{This work was supported by the BMBF under contract numbers 03MOPAF1 and 03KOPAF4}}
\author{Heiko Berninger\footnote{Section de Math\'ematiques, Universit\'e de Gen\`eve, 2-4 rue du Li\`evre, CP 64,
1211 Gen\`eve 4, Switzerland,
Heiko.Berninger@unige.ch}, Mario Ohlberger\footnote{Institute for Computational and Applied Mathematics, University of Muenster,
Einsteinstr. 62, 48149 Muenster, Germany,
mario.ohlberger@uni-muenster.de}, Oliver Sander\footnote{Institut f\"ur Geometrie und Praktische Mathematik, RWTH Aachen University,
Templergraben~55, 52062 Aachen, Germany,
sander@igpm.rwth-aachen.de}, Kathrin Smetana\footnote{Institute for Computational and Applied Mathematics, University of Muenster,
Einsteinstr. 62, 48149 Muenster, Germany,
kathrin.smetana@uni-muenster.de}}

\begin{document}

\maketitle

\begin{abstract}
We analytically and numerically analyze groundwater flow in a homogeneous soil described by the Richards equation, coupled to surface water represented by a set of ordinary differential equations (ODE's) on parts of the domain boundary, and with nonlinear outflow conditions of Signorini's type. The coupling of the partial differential equation (PDE) and the ODE's is given by nonlinear Robin boundary conditions. This article provides two major new contributions regarding these infiltration conditions. First, an existence result for the continuous coupled problem is established with the help of a regularization technique. Second, we analyze and validate a solver-friendly discretization of the coupled problem based on an implicit--explicit time discretization and on finite elements in space. The discretized PDE leads to convex spatial minimization problems which can be solved efficiently by monotone multigrid. Numerical experiments are provided using the {\sc Dune} numerics framework. 
\end{abstract}
\vspace*{\baselineskip}
\noindent
{\bf Keywords:} saturated-unsaturated porous media flow, Kirchhoff transformation, convex
minimization, finite elements, monotone multigrid, nonlinear transmission
problem\\

\noindent
{\bf AMS Subject Classification:}  35K61, 65N30, 65N55, 76S05\\
\newpage
\section{Introduction}

This article is concerned with existence of solutions, and efficient numerical approximation of unsaturated groundwater flow 
with surface water and nonlinear in- and outflow conditions. To start with, let us introduce the considered model.
Let \mbox{$\Omega \subset \mathbb{R}^d$,} $d=1,2,3$ denote a bounded domain occupied by a homogeneous soil 
with boundary $\partial \Omega$ of class $C^{1}$ and outer normal ${\boldsymbol \nu}$. 
Let $\Sigma_\text{in}, \Sigma_\text{out}$ and $\Sigma_{N} \subset \partial \Omega$ denote three relatively open, 
pairwise disjoint $d-1$-dimensional $C^{1}$-manifolds, representing the infiltration, outflow and Neumann boundaries,
such that $\partial \Omega = \overline{\Sigma_\text{in}} \cup \overline{\Sigma_\text{out}} \cup \overline{\Sigma_\text{N}}$
 and  $\overline{\Sigma}_\text{in} \cap \overline{\Sigma}_\text{out} = \emptyset$.
The unsaturated groundwater flow is modelled in the time interval $[0,T]$ using Richards equation (see e.g. \cite{Bear88})
for the water saturation $s: \Omega \times [0,T] \rightarrow [0,1]$ and the water pressure $p: \Omega \times [0,T] \rightarrow \mathbb{R}$
\begin{equation}\label{richards}
 n \partial_{t} s - \divergence\left(k(s)\mu^{-1}( \nabla p + {\boldsymbol e} )\right) = f \quad \text{in $\Omega \times [0,T]$}. 
\end{equation}
\begin{figure}
\center
\includegraphics[width=0.4\textwidth]{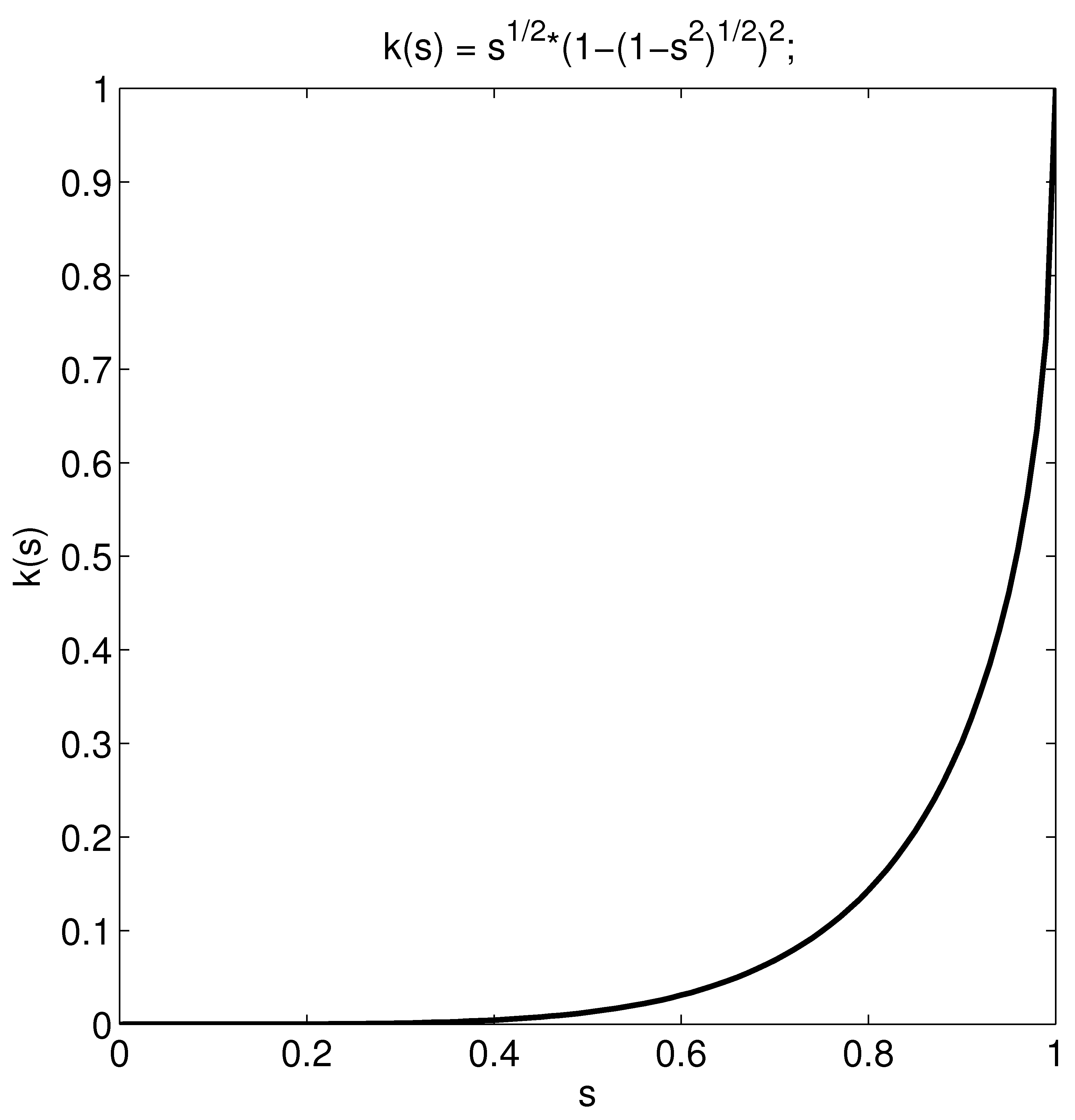} \quad \includegraphics[width=0.425\textwidth]{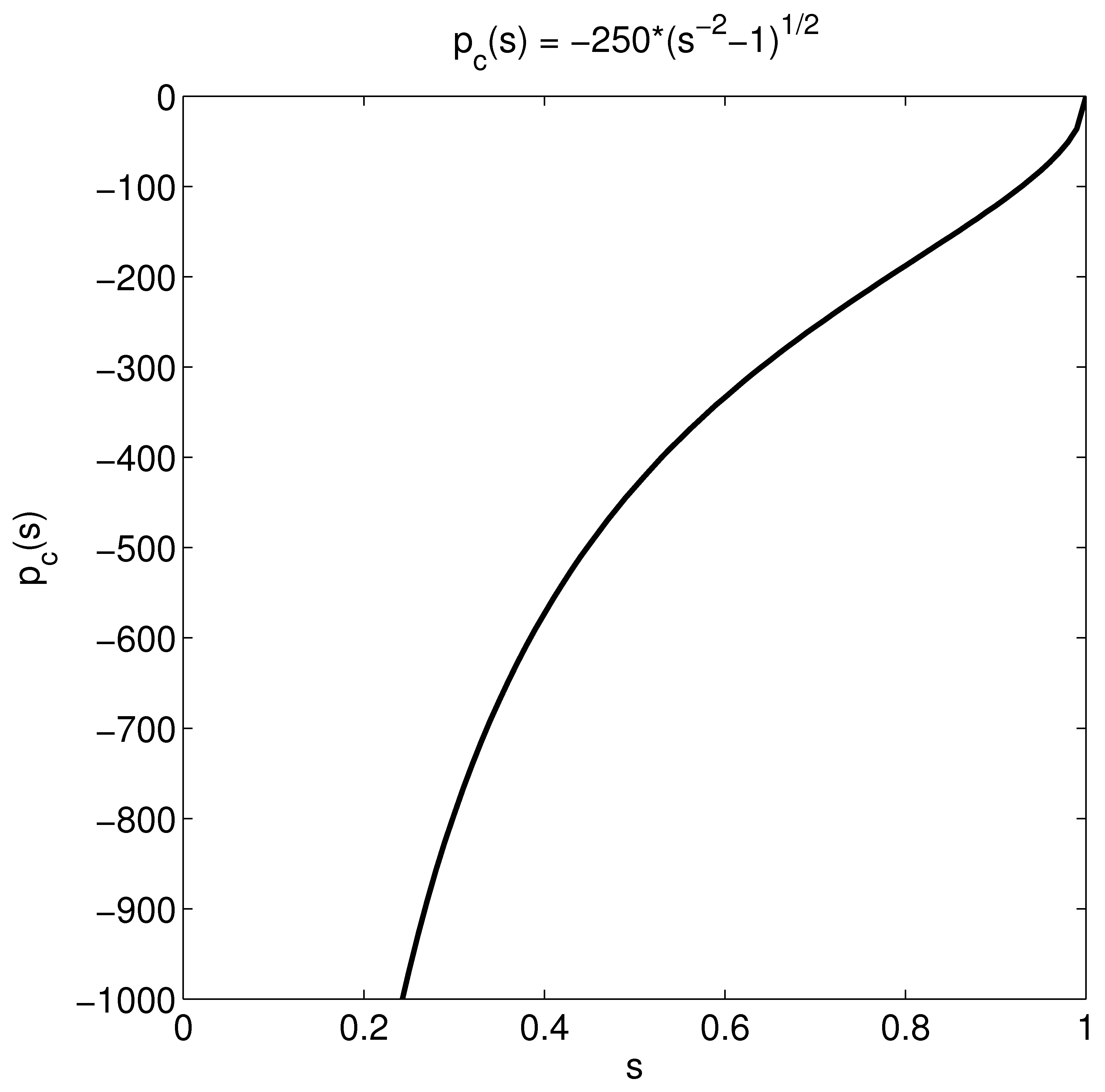} 
\caption{Typical shapes of the coefficient functions $k(s)$ and $p_{c}(s)$}
\label{pc_k_shape}
\end{figure}
Here $n$ denotes the porosity, $\mu$ the viscosity, $k$ the permeability, and ${\boldsymbol e}$ the gravity vector. 
For simplicity we set $n = \mu = 1$ and ${\boldsymbol e}=(0,0,1)^T$ in this paper. 
Due to our assumption of a homogeneous soil, the permeability $k$ depends only on the saturation $s$.  
In the Richards model, the air pressure in the pore space is assumed to be constant. We consider it normalized to 
$ p_\text{gas} = 0$ and replace the water pressure in \eqref{richards} by the capillary pressure function ${\tilde p}_{c}(s)$.  
Thus, the capillary pressure $p_{c}$ has to be seen as a multivalued graph with $p_{c}(s) = \{{\tilde p}_{c}(s)\}$ 
for $s < 1$ and $p_{c}(1) = [{\tilde p}_{c}(1), \infty)$. Furthermore, we assume the residual saturation, meaning the 
saturation level below which no flow of water occurs, to be zero, such that $p_c$ is defined on $[0,1]$. 
This is necessary for the definition of the infiltration boundary conditions. Additionally, we suppose ${\tilde p}_{c}(1) = 0$, 
which is essential for the limiting process on the outflow boundary. Typical shapes of $k(s)$ and $p_{c}(s)$ are depicted 
in Fig.~\ref{pc_k_shape}. It can be seen that the coefficient functions 
degenerate in the sense that $k(0) = 0$ and that $k$ and $p_{c}$ may have slopes that are unbounded in the neighbourhood 
of $s = 1$ and $s = 0$, respectively. These degeneracies of the coefficient functions are one of the major challenges when 
treating \eqref{richards} analytically.

We impose three different types of boundary conditions. On $\Sigma_N$  we prescribe homogeneous Neumann boundary 
conditions. On the outflow boundary $\Sigma_\text{out}$ we choose the Signorini-type boundary conditions
\begin{equation}\label{outflow}
{\boldsymbol q} \cdot {\boldsymbol \nu} \geq 0, \quad p \leq 0, \quad ({\boldsymbol q} \cdot {\boldsymbol \nu})\, p = 0\quad\mbox{on}\enspace\Sigma_\text{out}\times [0,T]
\end{equation}
with ${\boldsymbol q}=-k(s) {{}} ( \nabla p + {\boldsymbol e} )$ the Darcy velocity of the fluid.
This models a hydrophilic porous medium that is in contact with air (cf. \cite{OhlbergerSchweizer07}).
The inequalities \eqref{outflow} express that (i)~water may not enter the domain, 
(ii)~the pressure cannot be positive as $p_\text{air}$ has been normalized to $0$, 
and (iii)~water can exit only if $p = 0$. 
Concerning the analysis given below, a major problem is to give sense to the product of traces 
in the last equality of \eqref{outflow}. 
On the infiltration boundary $\Sigma_\text{in}$, we only consider inflow of water into the soil and 
ponding of water on the surface (without runoff).
We assume that there is an external water source $r(x,t), x \in \Sigma_\text{in}$.  The water at $x$ either infiltrates
directly into the soil, or accumulates in $x$ (ponding).  Hence the state of the surface water
can be described by the water table $w: \Sigma_\text{in} \times [0,T] \to \mathbb{R}$.
At the absence of surface water runoff, $w$ is
modelled as an ensemble of ordinary differential equations
\begin{equation}\label{ode1}
\partial_{t} w = {\boldsymbol q} \cdot \boldsymbol{\nu} + r\quad\text{on $\Sigma_\text{in}\times [0,T]$},
\end{equation}
which is coupled to the subsurface boundary flux $\boldsymbol{q}$ (to be understood in a suitable
weak sense).
Note that although $\Sigma_\text{in}$ is the only part of the boundary through which inflow can occur, we may also have outflow through $\Sigma_\text{in}$.
A possible model for the flux ${\boldsymbol q}$ on $\Sigma_\text{in}$ is to consider it as a vertical leakage through a very small semipervious layer of thickness $b$ and hydraulic conductivity $K_h$~\cite{FiloLuck99}, which can be modelled (cf.~\cite[Chapt.~6.4]{Bear88}) as
\begin{equation}\label{semipervious}
{\boldsymbol q} \cdot {\boldsymbol \nu} = \frac{p - w}{c}.
\end{equation}
Here $c = b/K_h$ is the resistance of the semipervious layer.  Its inverse $c^{-1}$ is also known as leakage coefficient. Unfortunately, \eqref{semipervious} is only valid for $w > 0$, as in the case of $w = 0$ and a negative pressure~$p$, we would obtain an inflow into the ground without water present on the infiltration boundary. Following~\cite{FiloLuck99} we
modify \eqref{semipervious} by a $w$-dependent factor to obtain a new law valid for all $w$
\begin{equation}\label{q_inflow}
{\boldsymbol q} \cdot {\boldsymbol \nu} = \frac{p_{+} + p_{-}\min\bigl\{1,\bigl(\frac{w}{\sigma}\bigr)_{+}\bigr\} - w {{}} }{c},
\end{equation}
where $p_{+} = \max\{p,0\}$ and $p_{-} = \min\{p,0\}$. Now, $w = 0$ in \eqref{q_inflow} implies ${\boldsymbol q} \cdot {\boldsymbol \nu}= 0$ if $p\leq 0$. The parameter $\sigma$ is a regularization threshold, with~\eqref{q_inflow} being the same as \eqref{semipervious} for $w\geq\sigma$. As we will see later, solutions of this model fulfill $w\ge 0$
automatically.

Existence results for unsaturated flow in porous media go back to the pioneering work 
of Alt, Luckhaus, and Visintin \cite{AltLuckhausVisintin84}.
The basis of their approach - which we will also use in this article - is
the so called Kirchhoff transformation. 
\cite{Schweizer07} 
picks up the methods used in \cite{AltLuckhausVisintin84} and proves the
existence of solutions 
for \eqref{richards} by regularization techniques, with a stronger solution
concept than in \cite{AltLuckhausVisintin84}. The vital improvement 
is that \cite{Schweizer07} allows dry regions, meaning regions where the saturation 
lies below the (positive) residual saturation, on the outflow boundary. 
Hence, one has to modify the boundary conditions, as the pressure is not defined in dry regions.
In \cite{Schweizer07} Robin boundary conditions are used in the regularized
problem and convergence to the limit on the outflow
boundary is obtained via defect measures.
Without considering outflow conditions,
\cite{ChenFriedKimu94} also studies dry regions and
proves the continuity of the free boundaries between the saturated 
and unsaturated regions and between the unsaturated and dry regions. 
In \cite{Otto97} uniqueness of the solution is proved for the problem introduced
in \cite{AltLuckhausVisintin84} by adapting the methods developed
in~\cite{Otto96}. In order to show $L^{1}$-contraction and with this the
uniqueness of the solution, doubling of variables techniques are
used as introduced in \cite{Kruzkov70}.
With respect to the condition on the infiltration boundary, in \cite{FiloLuck99}
condition~\eqref{q_inflow} is used to couple the Richards equation with a
hyperbolic PDE modelling ponding of water and surface runoff on the
infiltration boundary. Under the restriction $\partial_{t}s \in
L^{1}(\Omega)$ existence and uniqueness is shown for this coupled system. We
remark that we do not need this assumption in our approach.

Concerning the numerical treatment of the Richards equation, a rich
literature can be found. 
For an overview we refer to \cite[Sec.\;2.2]{BerningerDiss} and
\cite{BerningerKornhuberSander11} and the literature cited therein.  The
Richards equation has been discretized using finite volume methods
\cite{FuhrmannLangmach01,EymardHilhorst99}, mixed finite element
methods~\cite{ArbogastWheelerZhang96,KnabnerRaduPop04}, finite elements
\cite{Fuhrmann94,ForsythKropinski97}, and discontinuous Galerkin
schemes~\cite{Bastianuac05,Dawson2006}.  In most cases, the resulting
algebraic systems
were solved using Newton's method, which, however, suffers from ill-conditioning problems due to
the degeneracies in the parameter functions (Fig.~\ref{pc_k_shape}).
In contrast, \citet{BerningerKornhuberSander11} discretize the Kirchhoff-transformed Richards equation in a way that 
the resulting spatial problems can be solved efficiently using a monotone multigrid method.
Since this approach is based on convexity rather than on smoothness, it is also
well-suited for non-smooth Signorini-type boundary conditions. Numerical studies
in 
\cite{BerningerKornhuberSander11} demonstrate the efficiency of the solver, as
well as robustness for extreme soil parameters which correspond to parameter
functions that degenerate into step functions. An extension of this approach to 
heterogeneous soil using domain decomposition techniques can be found in~\cite{BerningerSander10}.
Concerning the numerical coupling of Richards equation with surface water,
we mention \cite{SochalaErnPiperno2009,Dawson2006,Dawson2008}, where
shallow water equations are used as the surface water model. Although 
different discretizations are applied, all these approaches enforce continuity of
pressure and normal flux, i.e., mass conservation, across the interface. These
coupling conditions have also been considered in an approach based on the
Kirchhoff transformed subsurface flow in~\cite{Bastian_et_al:2012}
and~\cite{BernKornSand11}, where an implicit time discretization of the coupled subsurface and 
surface water models has been solved with a heterogeneous domain decomposition
method. In \cite{BernKornSand11}, however, the pressure continuity was replaced by a
leakage condition such as \eqref{semipervious} that represented clogging of a
nearly impermeable river bed. Accordingly, a Robin--Neumann type iteration was
applied in~\cite{BernKornSand11}, whereas in~\cite{Bastian_et_al:2012} a
Dirichlet--Neumann type iteration was used to solve the coupled system. In the
present article, we choose an explicit time discretization for the surface water
as in \cite[Sec.\;4.3]{BerningerDiss}.

Let us now give an outline of the article.
In the first part of this article we prove a new existence result for the Richards equation~\eqref{richards} endowed 
with nonlinear outflow conditions of Signorini's type~\eqref{outflow} and coupled to surface water by nonlinear Robin conditions (\ref{ode1}, \ref{q_inflow}). The proof is done without requiring the assumption $\partial_{t}s \in L^{1}(\Omega)$. To achieve this goal, in Section~\ref{system_ass} we introduce a global pressure~$u$ with the help of the Kirchhoff transformation and reformulate the system accordingly. In Section~\ref{sec:existence}, we state the main existence theorem and derive its proof by combining two central approaches for existence and uniqueness results for nonlinear, degenerated 
parabolic PDEs, namely the regularization technique (cf.~\cite{OhlbergerSchweizer07,Schweizer07}) and the method 
of $L^{1}$-contraction (cf.~\cite{Otto97,FiloLuck99}).

In the second part of the article, starting with Section~\ref{discreteConvex}, we present our 
discretization of the coupled nonlinear problem that is implicit--explicit in time and uses finite elements in space. 
Via variational inequalities, the implicit--explicit time discretization leads to spatial convex 
minimization problems which contain nonlinear outflow and nonlinear physical
Robin boundary conditions but no Dirichlet conditions. We prove unique solvability of these problems if the
relative permeability is non-degenerate (e.g., regularized) and the water table
is non-zero.
In Section~\ref{numerics}, we give a numerical example combining the coupled surface--subsurface problem
and Signorini outflow conditions.  We observe that the proposed time discretization leads to a stable
method for reasonable time step sizes.  We also observe that the surface water height remains nonnegative
as predicted by the theory, with the exception of certain singular points.
Our numerical approach proves to be useful even beyond the confines of the assumptions needed
for the existence result.

\section{Global pressure formulation of the fully coupled system and assumptions on the data}
\label{system_ass}

For a better assessment of the unsaturated flow equation, we use Kirchoff's transformation 
to obtain an equivalaent formulation for a global pressure $u$. We then summarize the resulting fully coupled 
system with nonlinear boundary conditions and detail all assumptions on the data that we will need
for our analysis. 

Let us define the Kirchoff transformation 
\begin{align}
\label{Kirchhoff}
 {\tilde \Phi}(s) \colonequals \underset{0}{\overset{p_{c}(s)}{\int}} k(p_{c}^{-1}(q))\, dq, \qquad 
 \Phi(s) = \begin{cases}
           {\tilde \Phi}(s) \enspace &\mbox{if} \enspace s < 1,\\
	   [ {\tilde \Phi}(1), \infty) \enspace &\mbox{if} \enspace s = 1,
           \end{cases}
\end{align}
which yields $\nabla \Phi(s) = k(s) \nabla p_{c}(s).$ We then define the global pressure 
function
\begin{align*}
u(x,t) \colonequals \begin{cases}
          {\tilde \Phi}(s(x,t)) \enspace &\mbox{if}\enspace s(x,t) < 1,\\
	  {\tilde \Phi}(1) + k(1)(p(x,t) - {\tilde p}_{c}(1)) &\mbox{if} \enspace s(x,t) = 1.
	  \end{cases}
\end{align*}
We thus have $\nabla u = k(s) \nabla p$ for a sufficiently smooth solution $(s,p)$ of \eqref{richards}. 
Therefore, we obtain as equivalent versions of the Richards equation
\begin{align}
 \nonumber n\partial_{t} s &= \mu^{-1}\divergence(\nabla u + k(s) \cdot \gamma{\boldsymbol e}) + f(s), \quad u \in \Phi(s),\\[-1.5ex]
&\label{richKirch}\\[-1.5ex]
 \nonumber \mbox{or} \enspace n\partial_{t} \Phi^{-1}(u) &= \mu^{-1}\divergence (\nabla u + k(\Phi^{-1}(u)) \cdot \gamma{\boldsymbol e}) + f(\Phi^{-1}(u)).
\end{align}
Conversely, the inverse Kirchhoff transformation $
\kappa^{-1}:u\mapsto p
$
can be defined as follows (cf.~\cite{BerningerKornhuberSander11}) 
\begin{align}
\label{inverseKirchhoff}
p = \kappa^{-1}(u) \colonequals \begin{cases}- h(u) \enspace &\mbox{if} \enspace u \leq 0,\\
 \frac{u_{+}}{k(1)} \enspace &\mbox{else} \enspace \end{cases},
 \qquad h(u)\,\colonequals\, \underset{u_{-}}{\overset{0}{\int}}\, \frac{1}{k(\Phi^{-1}(\tau))}\enspace d\tau.
\end{align}
With the global pressure formulation (\ref{richKirch}) we are now prepared to state our fully coupled system.

\subsection{Fully coupled system for unsaturated flow in porous media with infiltration and Signorini-type outflow conditions}

Let us define $\Omega_{T} = \Omega \times [0,T]$, $\Omega_{0} = \Omega \times \{0\}$ and $\Sigma_{i,t} = \Sigma_{i} \times [0,t]$ for $t\in [0,T]$ with $i = \text{in},\text{out},N$. With given initial values $s_0$ and $w_0$, we get the following system of differential equations for the unknown global pressure $u$ and surface water height $w$
\begin{alignat}{2}
\label{system}
\nonumber  u \in \Phi(s), \enspace \mbox{and} \enspace \partial_{t} s  &= \divergence (\nabla u + k(s) \cdot {\boldsymbol e}) + f(s) \quad &\quad&\text{in $\Omega_{T}$},\\
\nonumber - (\nabla u + k(s) \cdot {\boldsymbol e} ) \cdot {\boldsymbol \nu} &\geq 0 &&\mbox{on} \enspace \overline{\Sigma}_{\text{out},T}\cup \Sigma_{N,T},\\
\nonumber u \leq {\tilde \Phi}(1) \enspace &\mbox{and} \enspace k^{2}(s)s - k^{2}(1)1 \leq 0 &&\mbox{on} \enspace \Sigma_{\text{out},T},\\
0 \leq -((\nabla u + k(s) &\cdot {\boldsymbol e})\cdot {\boldsymbol \nu}) \cdot (k^{2}(s)s - k^{2}(1)1) &&\mbox{on} \enspace \Sigma_{\text{out},T},\\
\nonumber (\nabla u + k(s) \cdot {\boldsymbol e})\cdot {\boldsymbol \nu} &= 0 &&\mbox{on} \enspace \Sigma_{N,T},\\
\nonumber (\nabla u + k(s) \cdot {\boldsymbol e})\cdot {\boldsymbol \nu} &= {{}} g(u,w) &&\mbox{on} \enspace \Sigma_{\text{in},T},\\
\nonumber s &= s_{0} &&\mbox{in} \enspace \Omega_{0},\\
\nonumber\partial_{t} w &= (r - g(u,w)) &&\mbox{in} \enspace \Sigma_{\text{in},T},\\
\nonumber w &= w_{0} &&\mbox{on} \enspace \Sigma_{\text{in},0},
\end{alignat}
where
\begin{equation*}
g(u,w) = - \frac{u_{+}}{k(1)c} + \frac{w {{}} }{c} + \frac{h(u) \psi(w)}{c},  \quad \mbox{and}  \quad \psi(w) = \min \left \{ 1, \left (\frac{w}{\sigma} \right)_{+} \right \}.
\end{equation*}
As the mappings $s\mapsto k^{2}(s)s$ and $s \mapsto p_{c}(s)$ are both monotone functions, and as
$p_{\text{gas}} = 0$ and ${\tilde p}_{c}(1) = 0$, the outflow conditions stated in \eqref{outflow} and the outflow conditions prescribed in \eqref{system} are formally equivalent. 

It is also possible to allow ${\tilde p}_{c}(1) < 0$ as it is the case, e.g., for Brooks--Corey parameter functions~\cite{BrooksCorey64}. But then the outflow conditions in~\eqref{outflow} and the outflow conditions in \eqref{system} are no longer equivalent. We would have to deal either with the conditions in \eqref{outflow} or with similar conditions formed by replacing $p$ by $u$ in \eqref{outflow} as done, e.g, in \cite{BerningerDiss}. In such case it would be very difficult to give a meaning to the traces in the last equality in \eqref{outflow}, as it demands an a priori estimate for $\|\divergence({\boldsymbol q})\|_{L^{2}(\Omega)}$, which in turn necessitates an estimate for $\|\partial_{t}s\|_{L^{2}(\Omega)}$. In order to solve this problem, one can either assume $\|\partial_{t}s\|_{L^{2}(\Omega)}\leq C$ or $\Phi^{\prime}(s) \geq c_{1} > 0$. If one makes one of these assumptions, it is also possible to show an existence result for the system with original outflow conditions \eqref{outflow}, based on 
the results in \cite{PopSchweizer11}. 
Otherwise, it is only feasible to derive a weighted $L^{2}$-estimate for $\divergence({\boldsymbol q})$ \cite{Schweizer07}
 and the original outflow conditions have to be replaced by some surrogate ones as suggested in \cite{Schweizer07} and~\eqref{system}.

\subsection{Assumptions on the data}

In order to guarantee among others the well-definedness of the boundary conditions, the coefficient functions have to fulfill several assumptions. Additional assumptions are needed for the proof of our existence result (see Theorem \ref{existence}). 


The relative permeability and capillary pressure functions are supposed to fulfill the following assumption.

\begin{assumption}\label{ass1}
 The permeability 
\begin{equation*}
 k \in C^{1}([0,1],[0,\infty)) \enspace \mbox{is monotonically non-decreasing}
\end{equation*}
and $k(s) = 0$ for $s = 0$. The capillary pressure $p_{c}: (0,1] \rightarrow \{0,1\}^{\mathbb{R}}$ is a monotone graph given by a function 
\begin{equation*}
 {\tilde p}_{c} \in C^{1}((0,1),\mathbb{R}), \enspace \mbox{monotonically increasing}.
\end{equation*}
We set $p_{c}(s) = \{{\tilde p}_{c}(s)\}$ for $s \in (0,1)$ and $p_{c}(1) = [{\tilde p}_{c}(1), \infty).$ Furthermore, we suppose  ${\tilde p}_{c}(1) = 0$ (hydrophilic case). Finally, we assume that with a constant $c_{0} > 0$ the following conditions hold:
\begin{alignat}{2}
\partial_{s} p_{c}(s) &\geq 1/c_{0} &\qquad&\forall\, s \in (0,1),\nonumber\\
|\partial_{s} k(s)|^{2} & \leq c_{0}\, k(s)  &&\forall\, s \in (0,1),\nonumber\\[-1.55ex]
\label{pc_k_estimates}&&\\[-1.55ex]
k(s)|p_{c}(s)| + \sqrt{k(s)}\partial_{s}p_{c}(s) &\leq c_{0}  &&\forall\, s \in (0,1/2),\nonumber\\
(1-s) \sqrt{\partial_{s} p_{c}(s)} &\leq c_{0}  &&\forall\, s \in (1/2,1).\nonumber
\end{alignat}
In order to guarantee that the condition on the infiltration boundary is well-defined, we additionally presume that the coefficient functions are chosen in such a way that
\begin{equation}
\label{crucial_impli}
u_{-} \in L^{\infty}(\Omega_{T})\enspace\Longrightarrow p_{-} \in L^{\infty}(\Omega_{T}).
\end{equation}
\end{assumption}

Taking Assumption \ref{ass1} for granted, we can prove $u_{-} \in L^{\infty}(\Omega_{T})$. Together with \eqref{crucial_impli} this yields the well-definedness of the conditions on the infiltration boundary. 

To justify the above assumptions from a physical viewpoint, we check if they can be fulfilled for the van Genuchten model \cite{vanGenuchten80} of the coefficient functions 
\begin{equation*}
k(s) = s^{1/2} \cdot \bigl(1 - (1 - s^{1/m})^{m} \bigr)^{2}, \quad p_{c}(s) = - \frac{1}{\alpha} (s^{-1/m}-1)^{1/l},
\end{equation*}
with $m=1-1/l$. In Table~\ref{bodenparameter} some typical values for $k(1)\gamma/\mu$ --- the hydraulic conductivity at full saturation --- and two variable parameters $\alpha$ and $l$ for five different soil types are listed. The values are extracted from~\cite{vanGenuchten80}.
\begin{table}[t]
\begin{center}
\begin{tabular}{|c|c|c|c|}
\hline soil type & $k(1)\gamma/\mu$ $[\text{cm}/\text{day}]$ & $\alpha$ $[\text{cm}^{-1}]$ & $l$ $[-]$\\ \hline 
\hline Hygiene sandstone & 108.0 & 0.0079 & 10.4 \\
\hline Touchet Silt Loam G.E.3 & 303.0 & 0.005 & 7.09 \\
\hline Silt Loam G.E.3 & 4.96 & 0.00423 & 2.06 \\
\hline Guelph Loam (drying) & 31.6 & 0.0115 & 2.03 \\
\hline Guelph Loam (wetting) & - & 0.02 & 2.76 \\
\hline Beit Netofa Clay & 0.082 & 0.00152 & 1.17 \\
\hline
\end{tabular}
\caption{Typical values for $k(1)\gamma/\mu$, $\alpha$ and $l$ used for the van Genuchten parametrization~\cite{vanGenuchten80} for five different soil types.}\label{bodenparameter} 
\end{center}
\end{table}
Starting with the implication \eqref{crucial_impli}, we see that this holds true for $l < 3$. Concerning estimates \eqref{pc_k_estimates}, we observe that they are all fulfilled except for the second one. This is due to the fact that the description of the relative permeability $k(s)$ proposed by van Genuchten is bounded for $s = 1$, but on the other hand its derivative blows up. However, since unbounded derivatives of $k$ are hydrologically unrealistic, we can modify $k(s)$ for $s$ near $1$ so that all estimates \eqref{pc_k_estimates} are satisfied. We emphasize the importance of the third inequality in \eqref{pc_k_estimates}, which guarantees that $k$ declines fast enough for small $s$ to absorb the pressure $p_{c}$, which goes to $-\infty$ in this case. 

Another parametrization of $k(s)$ and $p_c(s)$ is given by Brooks and Corey~\cite{BrooksCorey64}, see also \cite{Burdine53} and~\cite{vanGenuchten80}. One can check that all estimates \eqref{pc_k_estimates} are satisfied for this parametrization, however, with regard to Assumption~\ref{ass1}, we always have $\tilde{p}_c(1)<0$. Moreover, it is a characteristic property of the Brooks--Corey functions that the crucial implication \eqref{crucial_impli} is never fulfilled. On the contrary, the inverse Kirchhoff transformation \eqref{inverseKirchhoff} is always ill-posed around the minimal global pressure $u_c=\Phi(0)$ with $\kappa^{-1}(u_c)=-\infty$, cf.~\cite[Sec.\;1.3]{BerningerDiss}. With regard to the analysis in physical variables, we can consider regularizations of Brooks--Corey functions which exhibit the non-degenerate case $k(s)>c_0>0$, for which implication \eqref{crucial_impli} is satisfied, cf.~\cite[Sec.\;1.4.3]{BerningerDiss} and Section~\ref{discreteConvex}. However, even though the 
numerics for the Richards equation is challenging for the degenerate Brooks--Corey functions, we use them in our numerical examples in Section~\ref{numerics} analogue to~\cite{BerningerKornhuberSander11}.\\

\label{ass_data}
Concerning the source term, 
we assume that $f: \Omega_{T} \times [0,1] \rightarrow \mathbb{R}$ is bounded, continuously differentiable in all entries, and monotonically decreasing in $s$. Moreover, $f$ shall satisfy $f(x,t,1) \leq 0$ and $f(x,t,0) \geq 0$ for all $(x,t) \in \Omega_{T}$. Eventually, we suppose that $s \mapsto f(x,t,s)$ is affine on $(0,\lambda)$ for small $\lambda$. Furthermore, we assume that $r: \Sigma_{\text{in},T} \rightarrow \mathbb{R}^{+}$  is bounded and Lipschitz continuous. The initial conditions shall be given by functions $s_{0}: \Omega \rightarrow [0,1]$ and $w_0:\Sigma_\text{in}\to\R^+$, with $s_0 \in L^{\infty}(\Omega)$ and $w_0 \in L^{\infty}(\Sigma_\text{in})$ and
\begin{equation*}
s(x,0) = s_{0}(x) \quad \mbox{a.e. in} \enspace \Omega \quad \mbox{and} \quad w(x,0) = w_{0}(x) \quad \mbox{a.e.\ on} \enspace \Sigma_\text{in}.
\end{equation*}
We impose the following compatibility condition for $s_{0}$:
\begin{multline*}
 \exists\,p_{0}: \Omega \rightarrow \mathbb{R}, \enspace p_{0} \in p_{c}(s_{0}) \enspace \mbox{a.e.\ on} \enspace  \{k(s_{0}) > 0\}, \\
 p_{0} \in H^{1,2}(\Omega) \cap L^{\infty}(\Omega), \enspace s_{0}|_{\Sigma_\text{in}} > 0.
\end{multline*}

\section{Existence result for the degenerate coupled system}
\label{sec:existence}

The goal of this section is to prove an existence result for system \eqref{system}. Unfortunately, \eqref{system}
exhibits some unpleasant properties. First, the system is degenerate in the sense that $k(s)\to 0$ for $s\to 0$ and that the derivative $p_{c}'(s)$ may be unbounded in the neighbourhood of $s = 0$ and  $s = 1$. As a result, the solution of the Richards equation lacks regularity, making the treatment of the boundary conditions challenging, as it is difficult to give sense to the appearing traces of the functions. 
Moreover, we deal with a coupled system due to the boundary conditions at the infiltration boundary. 
To achieve our existence result, we thus first regularize the coefficient functions $k(s)$ and $p_{c}(s)$ 
with a small regularization parameter $\delta$ in order to obtain a non-degenerate system. 
Afterwards, we fix $\delta$ and decouple the system. Existence of solutions for the decoupled sub-problems 
then follows by standard arguments. By applying Banach's fixed
point theorem, we get a unique solution $(u_\delta,w_\delta)$ for the regularized coupled system
system, for fixed $\delta$. The crucial part when proving a priori estimates
independent of $\delta$ is the derivation of a maximum principle for
$u_{\delta}$ and $w_{\delta}$. This can be achieved by considering classical
solutions of a regularized parabolic system. Finally, we pass to the limit in
$\delta$ and get a solution for our original problem \eqref{system}. 
Our main result summarizes as follows.

\begin{theorem}[Existence of a solution pair $(u,w)$ of system \eqref{system}]\label{existence}
Let the assumptions from Section~\ref{system_ass} be fulfilled and let the pair $(u_{\delta},w_{\delta})$
be a sequence of regularized solutions in $L^{\infty}(\Omega_{T})\cap
L^{2}((0,T),H^{1,2}(\Omega)) \times C^{0}([0,T],L^{1}(\Sigma_\text{in}))\cap
L^{\infty}((0,T),L^{2}(\Sigma_\text{in}))$. Then there exists a subsequence
$(u_{\delta},w_{\delta})$ and a solution pair $(u,w)$ $\in
L^{2}((0,T),H^{1,2}(\Omega))$ $\times C^{0}([0,T],L^{1}(\Sigma_\text{in}))$ such
that 
\begin{align*}
u_{\delta} \rightharpoonup u \enspace &\mbox{in} \enspace L^{2}((0,T),H^{1,2}(\Omega)),\\
w_{\delta} \rightarrow w \enspace &\mbox{in} \enspace C^{0}([0,T],L^{1}(\Sigma_\text{in})),
\end{align*}
and $(u,w)$ are solutions of \eqref{system} in the distributional sense. Additionally, $w \geq 0$ holds a.e.\ on $\Sigma_\text{in}$ for all $t \in [0,T]$.
\end{theorem}

In the following subsections we will prove Theorem \ref{existence} as indicated above.

\subsection{Regularization of the coupled problem}\label{regul}

At first, we regularize the coefficient functions.

\begin{assumption}[Regularized coefficient functions]\label{ass2}
The regularized coefficient functions should fulfill the following conditions:
\begin{equation*}
k_{\delta} \in C^{1}([0,1],(0,\infty)) \quad \mbox{and} \quad \rho_{\delta} \in C^{0}([0,1], \mathbb{R}) \enspace \mbox{piecewise} \enspace C^{1}
\end{equation*}
are monotonically increasing. For $\delta \rightarrow 0$ we have $k_{\delta} \searrow k$ uniformly on [0,1], $\rho_{\delta} \rightarrow p_{c}$ uniformly on compact subsets of $(0,1)$ and $k_{\delta}(0) = \delta^{2}$. We suppose that $\bigcup_{\delta} \rho_{\delta}([0,1]) = \mathbb{R}$ and, for simplicity, that $p_{c}(1/2) = \rho_{\delta}(1/2)$ holds. Finally, $k_{\delta}$ and $\rho_{\delta}$ shall fulfill the inequalities of Assumption \ref{ass1}. 
\end{assumption}

The Assumptions \ref{ass1} and \ref{ass2} allow the definition of a regularized global pressure via the Kirchhoff transformation
\begin{equation*}
u_{\delta}(x,t) = \Phi_{\delta}(s) = \underset{0}{\overset{\rho_{\delta}(s)}{\int}} k_{\delta}(\rho_{\delta}^{-1}(q))\, dq,
\end{equation*}
and guarantee that $\Phi_{\delta}$ is bounded from below and that $\Phi_{\delta} \rightarrow \Phi$ uniformly on compact subsets of~$[0,1).$ We set $u_{c}= \Phi(0)$ as the minimal global pressure. 
\begin{example}
Let Assumption \ref{ass1} be fulfilled and suppose that $c_{1}s^{2} \leq k(s) \leq c_{2}s^{2}$ on $(0,1)$ for constants $0 < c_{1} \leq c_{2}$. Then the following regularization fulfills Assumption \ref{ass2}.
\begin{align*}
k_{\delta}(s) \colonequals \delta^{2} + k(s) \quad \forall \, s \in [0,1], \
\rho_{\delta}(s) \colonequals 
\begin{cases}
p_{c}(\delta) + \frac{s-\delta}{\delta} \quad        &\forall\, s \in [0,\delta], \\
p_{c}(s)                                    &\forall\, s \in (\delta,1-\delta], \\
p_{c}(1-\delta) + \frac{s - (1- \delta)}{\delta^{2}} &\forall\, s \in (1-\delta,1].
\end{cases}
\end{align*}
\end{example}

Note that all inequalities \eqref{pc_k_estimates}, which have to be satisfied by the regularized coefficient functions, except for inequality \eqref{pc_k_estimates}$_3$, are needed for the a priori estimates. If \eqref{pc_k_estimates}$_3$ is not fulfilled, we would still be able to prove convergence of the approximate solution pair to a limit $(u,w)$, but this limit would not necessarily solve our problem.

Using the regularized coefficient functions, we get the regularized Richards equation
\begin{equation*}
 \partial_{t}s_{\delta} = \divergence (\nabla u_{\delta} + k_{\delta}(s_{\delta}) \cdot {\boldsymbol e}) + f_{\delta}, \quad u_{\delta}= \Phi_{\delta}(s_{\delta}),
\end{equation*}
with the pressure $p_{\delta} = \rho_{\delta}(s_{\delta})$ and $f_{\delta} = f(x,t,s_{\delta}(x,t))$. In a straightforward manner we obtain the conditions on the Neumann and the infiltration boundary:
\begin{alignat*}{2}
(\nabla u_{\delta} + k_{\delta}(s_{\delta}) \cdot {\boldsymbol e})\cdot {\boldsymbol \nu} &=g_{\delta}(u_{\delta},w_{\delta}) &\qquad&\text{on $\Sigma_{\text{in},T}$}\\
(\nabla u_{\delta} + k_{\delta}(s_{\delta})\cdot {\boldsymbol e}) \cdot {\boldsymbol \nu} &= 0 &&\text{on $\Sigma_{N,T}$}.
\end{alignat*}
Inspired by \citep{Schweizer07}, we impose the following nonlinear Robin condition on the outflow boundary:
\begin{equation}\label{regout}
-(\nabla u_{\delta} + k_{\delta}(s_{\delta})\cdot {\boldsymbol e}) \cdot {\boldsymbol \nu} = \frac{1}{\delta} k_{\delta}(s_{\delta})(p_{\delta})_{+}.
\end{equation}
In this way, we keep the idea behind the original ouflow conditions, as water can flow out only if $p \geq 0$ and if $p > 0$ we have a very large outflow due to the scaling $1/\delta$. On the other hand, boundary condition \eqref{regout} is more handy than the original one, as we do not have to deal with a variational inequality. Rewriting \eqref{regout} with ${\mathbf q}_{\delta} = - (\nabla u_{\delta} + k_{\delta}(s_{\delta})\cdot {\boldsymbol e})$ one gets
\begin{equation*}
{\mathbf q}_{\delta}\cdot {\boldsymbol \nu} = k_{\delta}(s_{\delta}) \frac{(p_{\delta})_{+} - 0}{\delta},
\end{equation*}
which, similar to our condition on the infiltration boundary, can be interpreted as a pressure difference between the pressure inside the porous medium and the pressure outside $p_{\text{gas}} = 0$. The initial conditions are replaced by
\begin{align*}
s_{\delta}(x,0) = s^{\delta}_{0}(x) \colonequals \begin{cases}
                                        \rho_{\delta}^{-1}(p_{0}(x)) \enspace &\mbox{if} \enspace s_{0} > 1/2,\\
					s_{0}(x) &\mbox{if} \enspace s_{0} \leq 1/2,
                                       \end{cases}
\end{align*}
for all $x \in \Omega$, guaranteeing $\Phi_{\delta}(s_{0}^{\delta}) \in H^{1,2}(\Omega)$. 

We get the following regularized system for the unknowns $u_{\delta}$ and $w_{\delta}$
\begin{alignat}{2}
\label{regsys}
\nonumber  u_{\delta} = \Phi_{\delta}(s_{\delta}), \enspace \partial_{t} s_{\delta} &= \divergence ( \nabla u_{\delta} + k_{\delta}(s_{\delta})\cdot {\boldsymbol e} ) + f(s_{\delta})
 \quad &\qquad&\mbox{in} \enspace \Omega_{T}, \\
\nonumber -( \nabla u_{\delta} + k_{\delta}(s_{\delta})\cdot {\boldsymbol e} ) \cdot {\boldsymbol \nu} &= \frac{1}{\delta}k_{\delta}(s_{\delta})(p_{\delta})_{+}, &&\mbox{on} \enspace \Sigma_{\text{out}, T},\\
(\nabla u_{\delta} + k_{\delta}(s_{\delta}) \cdot {\boldsymbol e})\cdot {\boldsymbol \nu} &= 0 &&\mbox{on} \enspace \Sigma_{N, T}, \\
\nonumber(\nabla u_{\delta} + k_{\delta}(s_{\delta}) \cdot {\boldsymbol e})\cdot {\boldsymbol \nu} &=g_{\delta}(u_{\delta},w_{\delta}) &&\mbox{on} \enspace \Sigma_{\text{in}, T},\\
\nonumber s_{\delta}(x,0) &= s_{0}^{\delta}\colonequals
\begin{cases}
\rho_{\delta}^{-1}(p_{0}(x)) &\mbox{if} \enspace s_{0}(x) > \frac{1}{2}, \\
\nonumber s_{0}(x) &\mbox{if} \enspace s_{0}(x) \leq \frac{1}{2},
\end{cases} 
&&\mbox{in} \enspace \Omega_{0},\\
\nonumber 
\partial_{t}w_{\delta} &= (r - g_{\delta}(u_{\delta},w_{\delta})) &&\mbox{in} \enspace \Sigma_{\text{in}, T}, \\
\nonumber w_{\delta} &= w_{0} &&\mbox{on} \enspace \Sigma_{\text{in}, 0},
\end{alignat}
where
\begin{align*}
g_{\delta}(u_{\delta},w_{\delta}) &= - \frac{(u_{\delta})_{+}}{k_{\delta}(1)c} + \frac{w_{\delta}}{c} + \frac{h_{\delta}(u_{\delta})\psi(w_{\delta})}{c}, \\
h_{\delta}(u_{\delta}) &= \underset{(u_{\delta})_{-}}{\overset{0}{\int}} \frac{1}{k_{\delta}(\Phi_{\delta}^{-1}(\tau))} d\tau, \qquad \psi(w_{\delta})= \min\left\{1,\left(\frac{w_{\delta}}{\sigma}\right)_{+}\right\}.
\end{align*}

\subsection{Existence of a unique solution for fixed
\texorpdfstring{$\boldsymbol \delta$}{delta}}

Let $\delta$ be fixed. To prove the existence of a unique solution pair
$(u_{\delta},w_{\delta}) \in L^{2}((0,T),H^{1,2}(\Omega)) \times
C^{0}([0,T],L^{1}(\Sigma_\text{in}))$ of \eqref{system}, we take the complete
metric space $M= C^{0}([0,T],L^{1}(\Sigma_\text{in}))$ with the metric $\|w_{1}
- w_{2}\|_{M} =$ $\max_{0 \leq t \leq T}$ $\int_{\Sigma_\text{in}} |\,w_{1}(t) -
w_{2}(t)\,|\, dx$, construct a mapping $A: M \rightarrow M$ and show that this
mapping is $k$-contractive. The fixed point of the mapping $A$ is the solution
of the ODE-subproblem. Inserting this solution as data into the subproblem of
the Richards equation results in the solution pair of the whole system. For the
sake of readability, we omit the index $\delta$ in this subsection. At first, we
define the two subproblems.

\begin{definition}[Subproblem of the Richards equation]
Let $w_{j-1} \in M \cap L^{\infty}((0,T),L^{2}(\Sigma_\text{in}))$ be given for $j = 1,2,\ldots$. We characterize $u_{j}$ as the solution of the problem 
\begin{alignat}{2}\label{subrich}
\nonumber \partial_{t} \Phi^{-1}(u_{j}) - \divergence (\nabla u_{j} + k(\Phi^{-1}(u_{j})) \cdot {\boldsymbol e} ) &=   f(\Phi^{-1}(u_{j})) \quad &&\mbox{in} \enspace \Omega_{T}, \\
\nonumber -(\nabla u_{j} + k(\Phi^{-1}(u_{j})) \cdot {\boldsymbol e}) \cdot {\boldsymbol \nu} &= \frac{1}{\delta}\,k(\Phi^{-1}(u_{j}))\,(p_j)_{+} \quad &&\mbox{on} \enspace \Sigma_{\textnormal{out}, T},\\
(\nabla u_{j} + k(\Phi^{-1}(u_{j})) \cdot {\boldsymbol e}) \cdot {\boldsymbol \nu} &= 0 \quad &&\mbox{on} \enspace \Sigma_{N, T}, \\
\nonumber (\nabla u_{j} + k(\Phi^{-1}(u_{j})) \cdot {\boldsymbol e}) \cdot {\boldsymbol \nu} &= g(u_{j},w_{j-1}) \quad &&\mbox{on} \enspace \Sigma_{\textnormal{in}, T}, \\
\nonumber \Phi^{-1}(u_{j})(x,0) &= s_{0} \quad &&\mbox{in} \enspace \Omega_{0}, 
\end{alignat}
with
\begin{align*}
g(u_{j},w_{j-1}) &= - \frac{(u_{j})_{+}}{k(1)c} + \frac{w_{j-1}}{c} + \frac{h(u_{j})\psi(w_{j-1})}{c}, \\
h(u_{j}) &= \underset{(u_{j})_{-}}{\overset{0}{\int}} \frac{1}{k(\Phi^{-1}(\tau))}\,d\tau,
 \qquad \psi(w_{j-1})= \min\left\{1,\left(\frac{w_{j-1}}{\sigma}\right)_{+}\right\}.
\end{align*} 
\end{definition}

\begin{definition}[Weak formulation of the subproblem of the Richards equation]\label{weak_sol}
We say that $u_{j} \in L^{2}((0,T),H^{1,2}(\Omega))$ is a weak solution of problem \eqref{subrich}, if $\Phi^{-1}(u_{j}) \in L^{2}((0,T), L^{2}(\Omega))$ and
\begin{multline}\label{weak_rich}
 0 = \int_{\Omega} \Phi^{-1}(u_{j})(0,x)\,\varphi(0,x)\,dx + \int_{\Omega_{T}} \Phi^{-1}(u_{j})\partial_{t}\,\varphi\\
- \int_{\Omega_{T}}(\nabla u_{j} + k(\Phi^{-1}(u_{j}))\cdot {\boldsymbol e}) \cdot \nabla \varphi 
 + \int_{\Omega_{T}}f(\Phi^{-1}(u_{j}))\,\varphi\\
 - \int_{\Sigma_{\text{out}, T}} \frac{1}{\delta} (k(\Phi^{-1}(u_{j}))(p_{j})_{+})\,\varphi +\int_{\Sigma_{\text{in}, T}}\, g(u_{j},w_{j-1})\,\varphi
\end{multline}
for all test functions $\varphi \in L^{2}((0,T), H^{1,2}(\Omega))$ with $\partial_{t}\varphi \in L^{2}((0,T), L^{2}(\Omega))$ and $\varphi(T) = 0.$
\end{definition}

\begin{lemma}[Existence of a unique weak solution]
Let $w_{j-1} \in M \cap L^{\infty}((0,T),L^{2}(\Sigma_\text{in}))$ be given for $j \in \N$, and let the assumptions from Section~\ref{system_ass} and Subsection~\ref{regul} be fulfilled. Then there exists a unique weak solution $u_{j} \in L^{2}((0,T),H^{1,2}(\Omega))$ in the sense of Definition \ref{weak_sol}.
\end{lemma}
\begin{proof}
A proof can be found, e.g., in \cite{LadySoloUral00}.
\end{proof}

\begin{definition}[ODE subproblem]
 Let $w_{j-1} \in M\cap L^{\infty}((0,T),L^{2}(\Sigma_\text{in}))$ and $u_{j} \in L^{\infty}(\Omega_{T}) \cap L^{2}((0,T),H^{1,2}(\Omega))$ be given for $j = 1,2,\ldots$. We characterize $w_{j}$ as the solution of the problem
\begin{alignat}{2}
\nonumber\partial_{t}w_{j} &= r - g(u_{j},w_{j},w_{j-1})) &\qquad&\mbox{in} \enspace \Sigma_{\textnormal{in}, T} \\[-1.5ex]
&\label{subode}\\[-1.5ex]
\nonumber w_{j} &= w_{0}  &&\mbox{on} \enspace \Sigma_{\textnormal{in}, 0},
\end{alignat}
with
\begin{align*}
g(u_{j},w_{j},w_{j-1}) &= - \frac{(u_{j})_{+}}{k(1)c} + \frac{w_{j}}{c} + \frac{h(u_{j})\psi(w_{j-1})}{c}, \\
h(u_{j}) &= \underset{(u_{j})_{-}}{\overset{0}{\int}} \frac{1}{k(\Phi^{-1}(\tau))} d\tau,
 \qquad \psi(w_{j-1})= \min\left\{1,\left(\frac{w_{j-1}}{\sigma}\right)_{+}\right\}.
\end{align*} 
\end{definition}

\begin{lemma}[Existence of a unique solution of the ODE subproblem]
Let $w_{j-1} \in M\, \cap L^{\infty}((0,T),L^{2}(\Sigma_\text{in}))$ and $u_{j} \in L^{\infty}(\Omega_{T}) \cap L^{2}((0,T),H^{1,2}(\Omega))$ be given for $j \in \N$, and let the assumptions from Section~\ref{system_ass} and Subsection~\ref{regul} be fulfilled. Then there exists a unique solution $w_{j} \in M$ of the initial value problem~\eqref{subode}.
\end{lemma}
\begin{proof}
It can be easily proved that the right-hand side of the ODE \eqref{subode} fulfills the assumptions of the Carath{\'e}odory's existence theorem for ODEs, which in turn yields the claim. 
\end{proof}

We proceed with some estimates for the functions $u_{j}$ and $w_{j}$. 

\begin{coro}[A priori bounds for fixed $\delta$]
Under the assumptions from Section~\ref{system_ass} and Subsection~\ref{regul} and with a constant $C>0$, the following estimates for fixed $\delta > 0$ and arbitrary index $j$ hold true:
\begin{enumerate}
\item $s_{\delta} \in [0,1], \quad s_{\delta} \in L^{2}([0,T],H^{1,2}(\Omega)), \quad \int_{\Omega_{T}} |\partial_{t}s_{\delta}|^{2} \leq {C},$
\item $u_{\delta} \in L^{\infty}(\Omega_{T})\cap L^{2}([0,T],H^{1,2}(\Omega)),\quad \int_{\Omega_{T}} |\partial_{t}u_{\delta}|^{2} \leq {C},$
\item $\|\divergence(\nabla u_{\delta} + k_{\delta}(s_{\delta})\cdot {\boldsymbol e})\|_{L^{2}(\Omega_{T})} \leq {C},$
\item $p_{\delta}(s_{\delta}) \in L^{\infty}(\Omega_{T}), \quad p_{\delta}'(s_{\delta}) \in  L^{\infty}(\Omega_{T}),$
\item $h_{\delta}(u_{\delta}) \in L^{\infty}(\Sigma_{T}),$
\item $\Phi_{\delta}(s_{\delta})$ is Lipschitz continuous in $s_{\delta}$.
\end{enumerate}
\end{coro}

\begin{proof}
This is a simple conclusion from the estimates for arbitrary $\delta$, which will be derived in Subsection~\ref{apriori}. 
\end{proof}

Now we can define our mapping $A: M \rightarrow M$ as a composition of two mappings $B: M \rightarrow L^{2}((0,T),H^{1,2}(\Omega))$ and $D: M \times L^{2}((0,T),H^{1,2}(\Omega)) \rightarrow M$ with $u_{j} \colonequals Bw_{j-1}$ and $w_{j}\colonequals D(w_{j-1},u_{j})$, where $B$ and $D$ are defined by the subproblems \eqref{subrich} and \eqref{subode}. So $A$ can be defined as $w_{j} = D(w_{j-1},Bw_{j-1}) =: Aw_{j-1}$. We now prove the $k$-contractivity of $A$.

\begin{prop}[Estimate for the boundary conditions and the saturation]\label{bound_est}
Let $w_{j-1}, {\tilde w}_{j-1} \in M\cap L^{\infty}((0,T),L^{2}(\Sigma_\text{in}))$ be two solutions of the initial value problem \eqref{subode} and $u_{j}, {\tilde u}_{j} \in L^{\infty}(\Omega_{T}) \cap L^{2}((0,T),H^{1,2}(\Omega))$ the associated solutions of \eqref{subrich} for $j = 1,2,\ldots$. Furthermore, let the assumptions from Section~\ref{system_ass} and Subsection~\ref{regul} be fulfilled. Then for $0 \leq t \leq T$, $s_{j} = \Phi^{-1}(u_{j})$, ${\tilde s}_{j} = \Phi^{-1}({\tilde u}_{j})$ and a constant $C>0$ the following estimate holds true:
\begin{align}\label{estu}
&\int_{\Omega}|s_{j}(\cdot,t) - {\tilde s_{j}}(\cdot,t)| + \int_{\Sigma_{\text{in},t}} \frac{1}{k(1)c}|(u_{j})_{+} -({\tilde u_{j})}_{+}| + \frac{|\psi({\tilde w}_{j-1})|}{c}|h(u_{j}) - h({\tilde u_{j}})| \\
\nonumber &+\int_{\Sigma_{\text{out},t}} \frac{1}{\delta} |k(s_{j})(p_{+})(s_{j}) - k({\tilde s_{j}})(p_{+})({\tilde s_{j}})|
\leq {C}\|w_{j-1} - {\tilde w_{j-1}}\|_{L^{1}(\Sigma_{\text{in},t})}.
\end{align}
The constant $C$ can be specified as $C = c^{-1} \big(1 + \|h(u_{j})\|_{L^{\infty}(\Sigma_{\text{in},t})}\,L_{\psi}\big)$ where $L_{\psi}$ is the Lipschitz constant of $\psi$.
\end{prop}

\begin{proof}
To prove this estimate, we use duality techniques, which is done for a similar case in detail, e.g., in \cite{Filo96} and for the case here in \cite{Smetana08}. In short, we put all derivatives onto the test function $\varphi$ and solve classically a dual problem in~$\varphi$, which, loosely speaking, subtracts the boundary conditions appearing in the weak formulation, and adds them in absolute values, arranged with the right algebraic sign for the estimate.
\end{proof}

From the integral formulation of the ODE we get with straightforward modifications
\begin{align}\label{estw}
\int_{\Sigma_\text{in}}|w_{j} - {\tilde w_{j}}| &\leq \frac{1}{k(1)c}
\int_{\Sigma_{\text{in}, t}}|(u_{j})_{+} - ({\tilde u_{j})}_{+}| + \frac{1}{c}
\int_{\Sigma_{\text{in}, t}}  |w_{j} - {\tilde w_{j}}|\\ 
\nonumber & +{C}\,\int_{\Sigma_{\text{in}, t}}|w_{j-1} -{\tilde w_{j-1}}|  + \frac{1}{c}\,\int_{\Sigma_{\text{in}, t}}\psi({\tilde w_{j-1}})\,|h(u_{j}) - h({\tilde u_{j}})|,
\end{align}
with $C$ having been defined in Proposition \ref{bound_est}. This estimate \eqref{estw} together with \eqref{estu} give
\begin{align*}
\int_{\Sigma_\text{in}}\!\!\!|Aw_{j-1} - A{\tilde w_{j-1}}| =\!\!
\int_{\Sigma_\text{in}}\!\!\!|w_{j} - {\tilde w_{j}}|\,
\leq \frac{1}{c} \int_{\Sigma_{\text{in}, t}} \!\!\! |w_{j} - {\tilde w_{j}}|  +
\hat{C}\int_{\Sigma_{\text{in}, T}}\!\!\!|w_{j-1} - {\tilde w_{j-1}}|,
\end{align*}
where $\hat{C} = 2C$. The application of Gronwall's lemma yields
\begin{align*}
\int_{\Sigma_\text{in}}|w_{j} - {\tilde w_{j}}| &\leq \hat{C}
e^{\frac{t}{c}}\,\int_{\Sigma_{\text{in}, T}}|w_{j-1} - {\tilde w_{j-1}}| \leq
\hat{C} e^{\frac{T}{c}}\,\int_{\Sigma_{\text{in}, T}}|w_{j-1} - {\tilde
w_{j-1}}| \\
&\leq T  \hat{C} e^{\frac{T}{c}} \underset{0\leq t\leq
T}{\mathrm{max}}\int_{\Sigma_\text{in}}|w_{j-1} - {\tilde w_{j-1}}|,
\end{align*}
which results in
\begin{align}\label{finalest}
\|Aw_{j-1} - A{\tilde w_{j-1}}\|_{M} \leq T\hat{C} e^{\frac{T}{c}}
\|w_{j-1} - {\tilde w_{j-1}}\|_{M}.
\end{align}
For $T \cdot \hat{C}\cdot e^{\frac{T}{c}} = \theta < 1$, $A$ is obviously $k$-contractive. Since $w_{0} = {\tilde w_{0}}$, \eqref{finalest} provides the uniqueness of the solutions $w_{j}$ and $u_{j}$, too. With Banach Fixed Point Theorem we get existence of a unique solution pair $(u_{\delta},w_{\delta})$ of the regularized system \eqref{regsys}. We summarize this result in the following theorem.

\begin{theorem}\label{existence_reg}
For each $\delta > 0$ there exists a unique solution pair $u_{\delta} \in L^{\infty}(\Omega_{T}) \cap L^{2}((0,T),H^{1,2}(\Omega))$ and $w_{\delta} \in M \cap L^{\infty}((0,T),L^{2}(\Sigma_\text{in}))$ of the regularized system \eqref{regsys}. Furthermore, we have $w_{\delta} \geq 0$ a.e.\ on $\Sigma_\text{in}$ for all $t \in [0,T]$ and each $\delta > 0$.
\end{theorem}

\begin{proof}
It remains to prove that $w_{\delta} \geq 0$ a.e.\ on $\Sigma_\text{in}$ for all $t \in [0,T]$. We discretize the regularized system \eqref{regsys} explicitly in time, where we omit the index $\delta$ again. By $w^n$ and $u^{n}$ we denote the time-discrete approximation of $w(\cdot,t_n)$ and $u(\cdot,t_n)$ respectively for a time step $t_n\in [0,T]$, $n=0,1,\ldots,N$, $t_0\colonequals0$, $t_N\colonequals T$ and a time step size $\tau=t_{n+1}-t_n$. We show by means of mathematical induction that $w^{n}(x) \geq 0$ a.e on $\Sigma_\text{in}$ for all $n \in \mathbb{N}$. As $w^0(x) \geq 0$ by Assumptions on the data~\ref{ass_data}, we prove now the induction step $w^{n}(x) \geq 0 \Rightarrow w^{n+1}(x)\geq 0$ a.e. on~$\Sigma_\text{in}$. To this end, we consider for fixed $x \in \Sigma_\text{in}$ the scheme
\begin{align}\label{scheme_ODE}
\frac{w^{n+1}(x) - w^n(x)}{\tau} &= r(x,t_n) + \frac{u^{n}(x)_{+}}{k(1)c} - \frac{w^n(x)}{c} -\frac{h(u^{n}(x))\cdot\min\bigl\{1,\bigl(\frac{w^n(x)}{\sigma}\bigr)_{+}\bigr\}}{c}
\end{align}
for the solution of the ODE \eqref{subode} on $\Sigma_{\text{in},T}$. Please note that $u_{\delta}(x)$ is a solution of a non-degenerate elliptic problem and therefore regular enough to be evaluated on $\Sigma_\text{in}$. Let $w^{n}(x) \geq 0$ a.e on $\Sigma_\text{in}$ be fulfilled. Rearranging \eqref{scheme_ODE} leads to 
\begin{align*}
w^{n+1}(x) &= \left(1 - \frac{\tau}{c} \right)\, w^{n}(x) + \tau\, \underset{\geq 0}{\underbrace{\left(r(x,t_n) + \frac{u^{n}(x)_{+}}{k(1)c}\right)}} -  \tau \left( \frac{h(u^{n}(x))\psi(w^{n}(x))}{c} \right).
\end{align*}
Next we check which values $\tau$ guarantee that $w^{n+1}(x) \geq 0$, too. By simple modifications one gets that the estimate
\begin{equation}\label{estt}
\tau \leq \min\left\lbrace c, \frac{\sigma}{\frac{\sigma}{c} - r(x,t_n) + \frac{h(u^{n}(x))}{c}},\frac{c}{1 + \frac{h(u^{n}(x))}{\sigma}}\right\rbrace\quad\mbox{for}\enspace x\in\Sigma_\text{in}\enspace\mbox{a.e.}
\end{equation}
has to hold true if the second and third entry in the braces are nonnegative---otherwise these terms do not appear in \eqref{estt}. Now, if we let $\tau$ go to zero and use the global estimate
\begin{equation*}
\tau \leq  \underset{t \in [0,T]}{\mathrm{ess}\min} \left \{c, \frac{\sigma}{\frac{\sigma}{c} - r(x,t) 
+ \frac{h(u(x,t))}{c} },\frac{c}{1 + \frac{h(u(x,t))}{\sigma} } \right\}\quad\mbox{for}\enspace x\in\Sigma_\text{in}\enspace\mbox{a.e.},
\end{equation*}
we obtain $w_{\delta} \geq 0$ for all $\delta > 0$. Note that $r \in C^{0}(\bar{\Sigma}_{in,T})$ (\ref{ass_data}) and $u \in L^{\infty}(\Sigma_{in,T})$ imply that the introduced explicit Euler method converges pointwise a.e. on $\Sigma_{in}$. 
\end{proof}

We remark that estimate \eqref{estt} can in principle also be used for numerical simulations in order to guarantee that the numerical solution remains positive. For each time $t_{n}$, using an approximation of $u(\cdot,t_n)$ in \eqref{estt}, an upper bound $Dt$ for $\tau$ can be computed with \eqref{estt} ensuring that the approximation of $w(\cdot,t_{n+1})$ remains nonnegative for all $\tau \leq Dt$. However, numerical tests show that this bound is quite strict, see Section~\ref{numerics}.

\subsection{A priori bounds} \label{apriori}

In this subsection we derive a priori bounds for the solution pair $(u_{\delta},w_{\delta})$ of the regularized system \eqref{regsys} that are independent of $\delta$. Let the assumptions of Section~\ref{system_ass} and Subsection~\ref{regul} be fulfilled.

\begin{lemma}\label{apri}
There exists $C < \infty$ independent of $\delta > 0$ such that, for all $\delta > 0$, 
\begin{equation*}
\int_{\Omega_{T}} k_{\delta}\,\rho_{\delta}'\, |\partial_{t} s_{\delta}|^{2} \leq {C} \quad \mbox{and} \quad \enspace  \int_{\Omega_{T}} k_{\delta}(s_{\delta})\,|\,\divergence(\nabla u_{\delta} + k_{\delta}(s_{\delta}) \cdot {\boldsymbol e})|^{2} \leq {C}.
\end{equation*}
\end{lemma}

\begin{proof}
The first estimate is obtained by testing the weak formulation \eqref{weak_rich} with $\partial_{t} u$, 
using Assumption~\ref{ass2} and writing the integrands of the integrals on the boundary as total 
time derivatives. The second estimate follows directly from the first one.
For a detailed exposition of the proof we refer to \cite{Smetana08} and, in the case of 
Dirichlet boundary conditions on the infiltration boundary, to \cite{Schweizer07}. 
\end{proof}

\begin{theorem}[Maximum principle]\label{maximumprinciple}
There exists $C < \infty$ independent of $\delta$ such that, for all $\delta > 0$ 
\begin{equation*}
\|(u_{\delta})_{+}\|_{L^{\infty}(\Sigma_{\text{in},T})} + \|w_{\delta}\|_{L^{\infty}(\Sigma_{\text{in},T})} \leq C, \quad \mbox{and} \enspace \|u_{\delta}\|_{L^{\infty}(\Omega_{T})} \leq C.
\end{equation*}
Furthermore, there exists $p_{max} < \infty$ independent of $\delta$ such that $p_{\delta}(x,t) \leq p_{max}$ for almost all $(x,t) \in \Omega_{T}$ and all $\delta > 0$.
\end{theorem}

For the derivation of a maximum principle for the solution pair $(u_{\delta},w_{\delta})$, we seize ideas of \cite{FiloLuck99}. We can obtain $L^{\infty}$-estimates independent of $\delta$ in the case where the pair $(u_{\delta},w_{\delta})$ can be approximated by smooth functions $(u_{\delta}^{\varepsilon}, w_{\delta}^{\varepsilon})$ which are solutions of a regularized parabolic system. For the sake of simplicity we assume that the boundary of $\Omega$ is of class $C^{2}$. If $\partial \Omega$ is not smooth enough one can proceed as in \cite{FiloLuck99} and require that the elevation of the surface can be smoothly and periodically extended to the whole $\mathbb{R}^{2}$. Furthermore, we presume for the proof of Theorem \ref{maximumprinciple} that $k_{\delta} \in C^{2}([0,1],(0,\infty))$ and  $\rho_{\delta} \in C^{2}([0,1],\mathbb{R})$, otherwise an additional regularization step has to be performed. We consider the parabolic system
\begin{alignat}{2}
\label{smoothsys}
\nonumber u_{\delta}^{\varepsilon} = \Phi_{\delta}(s_{\delta}^{\varepsilon}), \enspace \partial_{t} s_{\delta}^{\varepsilon} &= \divergence ( \nabla u_{\delta}^{\varepsilon} + k_{\delta}(s_{\delta}^{\varepsilon})\cdot {\boldsymbol e} ) + f^{\varepsilon}(s_{\delta}^{\varepsilon})
 &\qquad&\mbox{in} \enspace \Omega_{T}, \\
\nonumber -( \nabla u_{\delta}^{\varepsilon} + k_{\delta}(s_{\delta}^{\varepsilon})\cdot {\boldsymbol e} ) \cdot {\boldsymbol \nu} &= \frac{1}{\delta}k_{\delta}(s_{\delta}^{\varepsilon})\xi^{\varepsilon}(p_{\delta}^{\varepsilon})\chi_{\varepsilon}, &&\mbox{on} \enspace \Sigma_{\text{out}, T},\\
\nonumber (\nabla u_{\delta}^{\varepsilon} + k_{\delta}(s_{\delta}^{\varepsilon}) \cdot {\boldsymbol e})\cdot {\boldsymbol \nu} &= 0 &&\mbox{on} \enspace \Sigma_{N, T}, \\
(\nabla u_{\delta}^{\varepsilon} + k_{\delta}(s_{\delta}^{\varepsilon}) \cdot {\boldsymbol e})\cdot {\boldsymbol \nu} &=-\frac{\xi^{\varepsilon}(u_{\delta}^{\varepsilon})}{k_{\delta}(1)c} +\frac{w_{\delta}^{\varepsilon}}{c}+\frac{h^{\varepsilon}(u_{\delta}^{\varepsilon})\psi^{\varepsilon}(w_{\delta}^{\varepsilon})}{c} &&\mbox{on} \enspace \Sigma_{\text{in}, T},\\
\nonumber s_{\delta}^{\varepsilon}(x,0) &= s_{0,\delta}^{\varepsilon}
&&\mbox{in} \enspace \Omega,\\
\nonumber 
\partial_{t}w_{\delta}^{\varepsilon} -\varepsilon \Delta w_{\delta}^{\varepsilon}&= r^{\varepsilon} + \frac{\xi^{\varepsilon}(u_{\delta}^{\varepsilon})}{k_{\delta}(1)c} - \frac{w_{\delta}^{\varepsilon}}{c} - \frac{h^{\varepsilon}(u_{\delta}^{\varepsilon})\psi^{\varepsilon}(w_{\delta}^{\varepsilon})}{c} &&\mbox{in} \enspace \Sigma_{\text{in}, T}, \\
\nonumber \nabla w_{\delta}^{\varepsilon} \cdot \nu_{2} &= 0 &&\hspace*{-1em}\mbox{on} \enspace \partial \Sigma_\text{in} \times [0,T],\\
\nonumber w_{\delta}^{\varepsilon} (x,0)&= w_{0,\delta}^{\varepsilon}  &&\mbox{in} \enspace \Sigma_\text{in},
\end{alignat}
where $\nu_{2}$ is the outer normal of $\partial \Sigma_\text{in}$ and the regularized functions fulfill the following assumptions. $f^{\varepsilon}(s,\cdot)$ and $f^{\varepsilon}(\cdot,x,t)$ shall be in $C^{2}(\Omega_{T})$ and $C^{2}([0,1])$ respectively, $f^{\varepsilon} \rightarrow f$ uniformly on $[0,1]\times \Omega_{T}$ and $\|f^{\varepsilon}\|_{C^{0}([0,1]\times\Omega_{T})} \leq \|f\|_{C^{0}([0,1]\times\Omega_{T})}$. $r^{\varepsilon}$ shall be in $C^{2}(\Omega_{T})$, $\|r^{\varepsilon}\|_{L^{\infty}(\Omega_{T})} \leq \|r\|_{L^{\infty}(\Omega_{T})}$ and $\|r^{\varepsilon} - r\|_{H^{1,2}(\Omega_{T})} \rightarrow 0$ as $\varepsilon \rightarrow 0.$  Furthermore, we require that $\xi^{\varepsilon} \in C^{2}(\mathbb{R})$ and that $\xi^{\varepsilon}$ is nonnegative and nondecreasing such that $\xi^{\varepsilon}(p_{\delta}^{\varepsilon}) = (p_{\delta})_{+}$ if $p_{\delta} \in \mathbb{R}^{-} \cup (\varepsilon,\infty)$ and $|\xi^{\varepsilon}(p_{\delta}^{\varepsilon})| \leq 1 + \varepsilon^{-1}$ and analog for 
$u_{\delta}$. $\chi^{\varepsilon} \in C_{0}^{\infty}(\Sigma_\text{out})$ is a cut-off function with $\chi^{\varepsilon}=1$ if $\mathrm{dist}(x,\partial \Sigma_\text{out}) > \varepsilon$. Additionally, $h^{\varepsilon}$ shall be in $C^{2}(\mathbb{R})$, $h^{\varepsilon} \geq 0$, $h^{\varepsilon} = 0$ if $u_{\delta}^{\varepsilon}\geq 0$, $h^{\varepsilon \prime} \leq 0$ and $h^{\varepsilon} \rightarrow h$ uniformly on $\mathbb{R}$. Moreover, we require $\psi^{\varepsilon}$ to be in $C^{2}(\mathbb{R})$, to be nonnegative and that $\psi^{\varepsilon} \rightarrow \psi$ uniformly on $\mathbb{R}^{+}$. Finally, the regularized initial values $w_{0,\delta}^{\varepsilon}$ and $s_{0,\delta}^{\varepsilon}$ shall be in $C^{2}(\overline{\Sigma_\text{in}})$ or $C^{3}(\overline{\Omega})$ respectively with $\|w_{0,\delta}^{\varepsilon} - w_{0,\delta}\|_{L^{1}(\Sigma_\text{in})} \rightarrow 0$ and $\|s_{0,\delta}^{\varepsilon}-s_{0,\delta}\|_{L^{2}(\Omega)} \rightarrow 0$ as $\varepsilon \rightarrow 0$. They also have to fulfill 
suitable compatibilty conditions. 
\begin{lemma}[Existence of a smooth solution pair  $(u_{\delta}^{\varepsilon},w_{\delta}^{\varepsilon})$ of system \eqref{smoothsys}]
Suppose that the regularized functions in system \eqref{smoothsys} fulfill the just mentioned assumptions. Then there exists a solution pair $(u_{\delta}^{\varepsilon},w_{\delta}^{\varepsilon})$ with
$
u_{\delta}^{\varepsilon} \in C^{0}(\overline{\Omega_{T}}), \nabla u_{\delta}^{\varepsilon} \in C^{0}(\overline{\Omega_{T}}), w_{\delta}^{\varepsilon} \in C^{0}(\overline{\Sigma_{\text{in},T}}), \Delta w_{\delta}^{\varepsilon} \in C^{0}(\overline{\Sigma_{\text{in},T}}), \partial_{t} w_{\delta}^{\varepsilon} \in C^{0}(\overline{\Sigma_{\text{in},T}}). 
$
Furthermore, we have $s_{\delta}^{\varepsilon} \in C^{0}(\overline{\Omega_{T}})$ and $\nabla s_{\delta}^{\varepsilon} \in C^{0}(\overline{\Omega_{T}})$.
\end{lemma}
\begin{proof}
The regularized coefficient functions fulfill all necessary assumptions in the respective theorems in Section~5 of\cite{LadySoloUral00}, which yield the required result.
\end{proof}

For the solution $(u_{\delta}^{\varepsilon},w_{\delta}^{\varepsilon})$ we can now derive a maximum principle and obtain $L^{\infty}$-estimates independent of $\varepsilon$ and $\delta$. 
\begin{lemma}[Maximum principle for the approximating solution pair $(u_{\delta}^{\varepsilon},w_{\delta}^{\varepsilon})$]
For the solution pair $(u_{\delta}^{\varepsilon},w_{\delta}^{\varepsilon})$ of~\ref{smoothsys} we have the following inequalities
\begin{equation*}
\|w_{\delta}^{\varepsilon}\|_{L^{\infty}(\Sigma_{\text{in},T})} \leq C \enspace \mbox{and} \enspace \|(u_{\delta}^{\varepsilon})_{+}\|_{L^{\infty}(\Sigma_{\text{in},T})} \leq \max \{\varepsilon, c_{1} \|w_{\delta}^{\varepsilon}\|_{L^{\infty}(\Sigma_{\text{in},T})} + c_{2}\},
\end{equation*}
where the constants $C, c_{1}$ and $c_{2}$ do not depend on $\delta$ and $\varepsilon$.
\end{lemma}
\begin{proof}
The proof follows ideas presented in \cite{FiloLuck99}. We omit the indices $\delta$ and $\varepsilon$ in this proof. Let $u_{s} \in C^{2}(\overline{\Omega})$ be the solution of
\begin{alignat*}{2}
\Delta u_{s} + \nabla k(\Phi^{-1}(u_{s}))\cdot {\boldsymbol e} + F&=0 &\qquad&\mbox{in} \enspace \Omega,\\
 u_{s} &= 0  &&\mbox{on} \enspace \Sigma_\text{in},\\
(\nabla u_{s} + k(\Phi^{-1}(u_{s})) \cdot {\boldsymbol e}) \cdot {\boldsymbol \nu} &= 0  &&\mbox{on} \enspace \Sigma_{N},\\
(\nabla u_{s} + k(\Phi^{-1}(u_{s})) \cdot {\boldsymbol e}) \cdot {\boldsymbol \nu} &= -\delta^{-1}\xi^{\varepsilon}(u_{s}) \chi^{\varepsilon}  &&\mbox{on} \enspace \Sigma_\text{out},
\end{alignat*}
where $F = \|f(x,t,s)\|_{C^{0}(\Omega_{T}\times [0,1])}$. Analogous to \cite{FiloLuck99} one can show that $u_{s}$ is nonnegative. Therefore the term $\nabla  k(\Phi^{-1}(u_{s}))$ is unproblematic as $\Phi^{-1}_{\delta}(u_{s}^{\delta})$ tends to zero as $\delta \rightarrow 0$ for $u_{s}^{\delta} \geq 0$. Thus we can bound $\Delta u_{s} + \nabla k(\Phi^{-1}(u_{s}))\cdot {\boldsymbol e}$ and see that hence $u_{s}^{\delta} \rightarrow 0$ as $ \delta \rightarrow 0$ on $\Sigma_\text{out}$. A classical maximum principle yields that $\|u_{s}\|_{C^{1}(\overline{\Omega})} \leq C$, with $C$ independent of $\delta$ and $\varepsilon$.

We fix $t \in (0,T)$, define $m_{t}\colonequals\max_{(x,\tau) \in \overline{\Sigma_{\text{in},T}}}u_{+}(x,\tau)$ and set $U_{s}\colonequals u_{s}(x) + m_{t}$ for $x \in \overline{\Omega}$. With the same techniques as in \cite{AltLuckhaus83} one can prove that
\begin{equation}\label{supersol_1}
u(x,\tau) \leq U_{s}(x) \quad \forall (x,\tau) \in \overline{\Omega_{t}}.
\end{equation} 
We have that 
\begin{equation}\label{supersol_2}
u(x(\hat{t}),\hat{t}) = U_{s}(x(\hat{t})) \enspace \mbox{for some} \enspace (x(\hat{t}),\hat{t}) \in \overline{\Sigma_\text{in}} \times [0,t]
\end{equation}
with $m_{t} = u_{+}(x(\hat{t}),\hat{t})$. We suppose that $m_{t} > 0$, otherwise the proof is easy. Then \eqref{supersol_1} and \eqref{supersol_2} yield $\nabla U_{s}(x(\hat{t}),\hat{t}) \cdot {\boldsymbol \nu} \leq \nabla u(x(\hat{t}),\hat{t}) \cdot {\boldsymbol \nu}$, which in turn leads to
\begin{equation}\label{Abschaetzung_u}
\frac{\xi^{\varepsilon}(m_{t})}{k(1)c} \leq \frac{w(x(\hat{t}),\hat{t})}{c} - (\nabla u_{s}(x(\hat{t})) + k(\Phi^{-1}(u_{s}(x(\hat{t}))) \cdot {\boldsymbol e})\cdot {\boldsymbol \nu} \leq \frac{v(t)}{c} + c_{2},
\end{equation}
where $c_{2}$ does not depend on $\varepsilon$ and $\delta$ and
\begin{equation*}
v(t) \colonequals \underset{\tau \in [0,t]}{\max} W(\tau), \quad W(\tau) \colonequals \underset{x \in \overline{\Sigma_\text{in}}}{\max} w(x,\tau).
\end{equation*}
Then $v$ is nonnegative, nondecreasing and Lipschitz continuous on $[0,T]$. For $W(t) = w(x(t),t)$ with $x(t)$ in $\Sigma_\text{in}$ we have $\Delta w(x(t),t) \leq 0$. At the points where $W(t)$ is differentiable, there holds
\begin{equation}\label{absch_w}
W^{\prime}(\tau) \leq - \frac{W(\tau)}{c} + \|r\|_{L^{\infty}(\Sigma_{\text{in},T})} + \frac{\xi^{\varepsilon}(u(x(\tau),\tau))}{k(1)c} \leq \frac{\xi^{\varepsilon}(m_{t})}{k(1)c} - \frac{W(\tau)}{c} + C.
\end{equation}
At first we assume that $v(t) = W(\tau)$ for some $0 \leq \tau < t$. Since then $v(s) = W(\tau)$ for all $s \in [\tau,t]$, we have $v^{\prime} = 0$ on $(\tau,t)$. Thus let $v(t) = W(t)$ for some $t \in (0,T)$. Then \eqref{Abschaetzung_u} and \eqref{absch_w} lead to $v^{\prime}(t) \leq C$ and hence $v(t) \leq C$ for any $t \in [0,T]$ and a constant $C$ not depending on $\delta$ and $\varepsilon$. Therefore we have 
\begin{equation}\label{L_infty_w}
\|w^{\varepsilon}_{\delta}\|_{L^{\infty}(\Sigma_{\text{in},T})} \leq C
\end{equation}
and with \eqref{Abschaetzung_u} and \eqref{L_infty_w} also
\begin{equation*}
\|(u^{\varepsilon}_{\delta})_{+}\|_{L^{\infty}(\Sigma_{\text{in},T})} \leq \max\{\varepsilon, c_{1} \|w^{\varepsilon}_{\delta}\|_{L^{\infty}(\Sigma_{\text{in},T})} + c_{2}\}.
\end{equation*}

\end{proof}

If we now pass to the limit in $\varepsilon$ in \eqref{smoothsys} and obtain the following estimates. 

\begin{lemma}[$L^{\infty}$-estimates for $u_{\delta}$ and $w_{\delta}$]
For the solution pair $(u_{\delta},w_{\delta})$ of~\eqref{regsys} we have the following inequalities
\begin{align}
\|u_{\delta}\|_{L^{\infty}(\Omega_{T})} + \|w_{\delta}\|_{L^{\infty}(\Sigma_{\text{in},T})} &\leq C, \label{l_inf_1}\\
\|(u_{\delta})_{+}\|_{L^{\infty}(\Sigma_{\text{in},T})} &\leq c_{1}\|w_{\delta}\|_{L^{\infty}(\Sigma_{\text{in},T})} + c_{2}, \label{l_inf_2}
\end{align}
where the constants $C,c_{1}$ and $c_{2}$ do not depend on $\delta$.
\end{lemma}
\begin{proof}
First let us assume that $s_{0,\delta} \in L^{\infty}(\Omega)\cap H^{1,2}(\Omega)$, $w_{0,\delta} \in L^{\infty}(\Sigma_\text{in})\cap H^{1,2}(\Sigma_\text{in})$ and that they fulfill suitable compatibility conditions. Approximating $s_{0,\delta}$ and $w_{0,\delta}$ by the initial conditions of \eqref{smoothsys} with $s_{0,\delta}^{\varepsilon}\rightarrow s_{0,\delta}$ in $H^{1,2}(\Omega)$ and $w_{0,\delta}^{\varepsilon}\rightarrow w_{0,\delta}$ in $H^{1,2}(\Sigma_\text{in})$, we obtain for the solution pair $(u_{\delta}^{\varepsilon},w_{\delta}^{\varepsilon})$ of \eqref{smoothsys} the following a priori estimates
\begin{align}
 \label{en:1} &\|u^{\varepsilon}_{\delta}\|_{L^{\infty}(\Omega_{T})} \leq
C, \quad p_{\delta}^{\varepsilon} \leq p_{max} \enspace \mbox{on} \enspace
\Omega_{T}, \\
& \nonumber| \nabla u^{\varepsilon}_{\delta}\|_{L^{2}(\Omega_{T})} +
\|\partial_{t} u^{\varepsilon}_{\delta}\|_{L^{2}(\Omega_{T})} + \|\partial_{t}
s^{\varepsilon}_{\delta}\|_{L^{2}(\Omega_{T})}\\
& \label{en:2}\hspace*{2em} + \| \mathrm{div} (\nabla u^{\varepsilon}_{\delta} +
k_{\delta}(s^{\varepsilon}_{\delta}) \cdot {\boldsymbol e}
)\|_{L^{2}(\Omega_{T})} + \| \partial_{t} w^{\varepsilon}_{\delta}
\|_{L^{2}(\Sigma_{\text{in},T})} \leq C(\delta),\\
& \nonumber\underset{t \in [0,T]}{\max} \|\partial_{t}
w^{\varepsilon}_{\delta}(\cdot, t)\|_{L^{1}(\Sigma_\text{in})} + \underset{t \in
[0,T]}{\max} \| \nabla w^{\varepsilon}_{\delta}(\cdot , t)
\|_{L^{1}(\Sigma_\text{in})} \\
\label{en:3}&\hspace*{2em}\leq C(\delta) (1 + \|\partial_{t}
w^{\varepsilon}_{\delta}(\cdot, 0)\|_{L^{1}(\Sigma_\text{in})} + \| \nabla
w^{\varepsilon}_{\delta}(\cdot , 0) \|_{L^{1}(\Sigma_\text{in})}),
\end{align}
where all constants $C$ and $p_{max}$ are independent of $\varepsilon$ and the
constants in \eqref{en:1} do not depend on $\delta$ either. Here Lemma~3 in
\cite{Schweizer07} yields \eqref{en:1} and the estimates in \eqref{en:2}
can be obtained by a testing procedure.
 Applying the vanishing viscosity method (see for instance
\cite{BardosLeRouxNedelec}) results in  \eqref{en:3}.
We can extract subsequences $u^{\varepsilon}_{\delta} \rightarrow u_{\delta}$ in $L^{2}(\Omega_{T})$ with $\nabla u^{\varepsilon}_{\delta} \rightharpoonup \nabla u_{\delta}$ in $L^{2}(\Omega_{T})$ and $w^{\varepsilon}_{\delta}(\cdot ,t) \rightarrow w_{\delta}(\cdot ,t)$ in $L^{1}(\Sigma_\text{in})$ with $\partial_{t} w^{\varepsilon}_{\delta}(\cdot ,t) \rightharpoonup w_{\delta}(\cdot ,t)$ in $L^{1}(\Sigma_\text{in})$ for $t \in [0,T]$ as $\varepsilon \rightarrow 0$, where $(u_{\delta},w_{\delta})$ is the solution pair of \eqref{regsys}. To verify that the limit of $w^{\varepsilon}_{\delta}$ fulfills the ODE in \eqref{regsys}, we test with a function in $C^{\infty}_{0}(\Sigma_\text{in})$ and apply the fundamental lemma of calculus of variations to obtain that $\partial_{t}w_{\delta} -r_{\delta} - (u_{\delta})_{+}/(k_{\delta}(1)c) + w_{\delta}/c +(h_{\delta}(u_{\delta})\psi(w_{\delta}))/c = 0$ a.e. on $\Sigma_\text{in}$ for all $t \in [0,T]$.
Thus \eqref{l_inf_1} and \eqref{l_inf_2} follow. Let us now suppose that $(s_{0,\delta},w_{0,\delta})$ are only in $L^{\infty}(\Omega)$ or $L^{\infty}(\Sigma_\text{in})$ respectively. We approximate $(s_{0,\delta},w_{0,\delta})$ by $C^{\infty}_{0}$-functions $(s_{0,\delta}^{\eta},w_{0,\delta}^{\eta})$ such that $s_{0,\delta}^{\eta} \rightarrow s_{0,\delta}$ in $L^{1}(\Omega)$ and $w_{0,\delta}^{\eta}\rightarrow w_{0,\delta}$ in $L^{1}(\Sigma_\text{in})$ as $\eta \rightarrow 0$ and denote the corresponding solution pair of the system \eqref{regsys} $(s_{\delta}^{\eta},w_{\delta}^{\eta}).$ Estimates \eqref{estu} and \eqref{finalest} imply that $w_{\delta}^{\eta} \rightarrow w_{\delta}$ in $C^{0}([0,T];L^{1}(\Sigma_\text{in}))$ and $s_{\delta}^{\eta}(\cdot ,t) \rightarrow s_{\delta}(\cdot ,t)$ in $L^{1}(\Omega)$ for a.e. $t \in [0,T]$ as $\eta \rightarrow 0$. Moreover, \eqref{estu} renders 
$\int_{\Sigma_\text{in}} |(u_{\delta}^{\eta})_{+} - (u_{\delta})_{+}| + \psi(w_{\delta}) | h_{\delta}(u_{\delta}^{\eta}) - h_{\delta}(u_{\delta}) | \rightarrow 0$ as $\eta \rightarrow 0$. Thus a priori estimates \eqref{l_inf_1} and \eqref{l_inf_2}, which hold true for $(u_{\delta}^{\eta}, w_{\delta}^{\eta})$ also hold true for $(u_{\delta} , w_{\delta})$. 
\end{proof}

\begin{remark}\label{s_01}
Note that the parabolic maximum principle yields that $s_{\delta} \in [0,1] \enspace \mbox{a.e.\ in} \enspace \Omega_{T}$
(cf. \cite{Schweizer07} for details). 
\end{remark}

\begin{lemma}
There exists $C < \infty$ independent of $\delta$ such that, for all $\delta > 0$, 
\begin{equation}\label{grad_u_eq}
\int_{\Omega_{T}} k_{\delta}(s_{\delta})|\nabla s_{\delta}|^{2} \leq {C} \quad \mbox{and} \quad u \in L^{2}((0,T),H^{1,2}(\Omega)).
\end{equation}
\end{lemma}
\begin{proof}
To prove \eqref{grad_u_eq}, one first proves for a constant $C$, independent of $\delta$, the estimate
\begin{equation*}
\int_{\Omega_{T}} k_{\delta}(s_{\delta}) |\nabla p_{\delta}|^{2}\chi_{\{s_{\delta}\geq 1/2\}} + k_{\delta}(s_{\delta})\partial_{s}p_{\delta}(s_{\delta})|\nabla s_{\delta}|^{2}\chi_{\{s_{\delta}< 1/2\}} \leq C,
\end{equation*}
where $\chi_D$ denotes the characteristic function of a set $D$. 
To this end, we test the weak formulation \eqref{weak_rich} with suitable test functions. 
For $s_{\delta} < 1/2$ we choose $\eta_{-}(s_{\delta}) = (s_{\delta} - 1/2)_{-} + 1/2$ and 
for $s_{\delta} \geq 1/2$ we pick $\eta_{+}(p_{\delta}) = (p_{\delta} - \overline{p})_{+}$, 
where $ p_{\delta}(0) < \overline{p} \leq p_{\delta}(1/2)$. Theorem \ref{maximumprinciple} 
and Remark \ref{s_01} guarantee the boundedness of $\eta_{-}(s_{\delta})$ and $\eta_{+}(p_{\delta})$. 
For more details we refer to \cite{Smetana08,{Schweizer07}}.
\end{proof}

\subsection{Passing  to the limit in \texorpdfstring{${\boldsymbol \delta}$}{delta}}

\begin{theorem}
Let the assumptions of Section~\ref{system_ass} and Subsection~\ref{regul} be fulfilled and let $(u_{\delta}, w_{\delta})$ be the unique solution pair of the regularized system \eqref{regsys}. Then for a subsequence $\delta \rightarrow 0$ there holds
\begin{align}
s_{\delta} &\overset{*}{\rightharpoonup} s \qquad \,\,\mbox{in} \enspace L^{\infty}(\Omega_{T}),\label{s}\\
u_{\delta} &\rightharpoonup u \qquad \,\mbox{in} \enspace L^{2}([0,T],H^{1,2}(\Omega)),\label{u}\\
w_{\delta} &\rightarrow w \qquad \mbox{in} \enspace C^{0}([0,T], L^{1}(\Sigma_\text{in})),\nonumber
\end{align}
for appropriate limit functions $s:\, \Omega_{T} \rightarrow [0,1], \, u:\, \Omega_{T} \rightarrow \mathbb{R}$ and $w:\, \Sigma_{\text{in},T} \rightarrow \mathbb{R}$. The limits satisfy
\begin{alignat}{2} 
\label{bounda}
\nonumber u \in \Phi(s)\enspace \mbox{in} &\enspace \Omega_{T}, \enspace \partial_{t} s  = \divergence (\nabla u + k(s) \cdot {\boldsymbol e}) + f(s)  &\quad&\mbox{in} \enspace {\mathcal D}^{\prime}(\Omega_{T}),\\
\nonumber- (\nabla u + k(s) \cdot {\boldsymbol e} ) \cdot {\boldsymbol \nu}
&\geq 0 &&\hspace*{-2em}\mbox{on} \enspace \overline{\Sigma}_{\text{out},T} \cup
\Sigma_{N,T}\\
\nonumber u \leq {\tilde \Phi}(1) \enspace &\mbox{and} \enspace k^{2}(s)s - k^{2}(1)1 \leq 0 &&\mbox{on} \enspace \Sigma_{\text{out},T},\\
0 &\leq ((\nabla u + k(s) \cdot {\boldsymbol e})\cdot {\boldsymbol \nu}) \cdot (k^{2}(s)s - k^{2}(1)1) &&\mbox{on} \enspace \Sigma_{\text{out},T},\\
\nonumber (\nabla u + k(s) \cdot {\boldsymbol e})\cdot {\boldsymbol \nu} &= 0 &&\mbox{on} \enspace \Sigma_{N,T},\\
\nonumber(\nabla u + k(s) \cdot {\boldsymbol e})\cdot {\boldsymbol \nu} &= g(u,w) &&\mbox{on} \enspace \Sigma_{\text{in},T},\\
\nonumber s &= s_{0} &&\mbox{in} \enspace \Omega_{0},\\
\nonumber \partial_{t} w &= (r - g(u,w)) &&\mbox{in} \enspace \Sigma_{\text{in},T},\\
\nonumber w &= w_{0} &&\mbox{on} \enspace \Sigma_{\text{in},0},
\end{alignat}
where
\begin{align*}
g(u,w) &= - \frac{u_{+}}{k(1)c} + \frac{w}{c} + \frac{h(u) \psi(w)}{c}, \\
h(u) &= \underset{u_{-}}{\overset{0}{\int}} \frac{1}{k(\Phi^{-1}(\tau))}\, d\tau,
 \qquad \psi(w) = \min \left \{ 1, \left (\frac{w}{\sigma} \right)_{+} \right \}.
\end{align*}
The traces in \eqref{bounda}$_3$--\eqref{bounda}$_7$ exist in the sense of distributions. Furthermore, there holds $w \geq 0$ a.e.\ on $\Sigma_\text{in}$ for all $t \in [0,T]$.
\end{theorem}

\begin{proof}
The assumptions in Section~\ref{system_ass} and Subsection~\ref{regul} and the a priori bounds of Subsection~\ref{apriori} yield \eqref{s}, \eqref{u}, $s: \Omega_{T} \rightarrow [0,1]$ and $f_{\delta} \rightharpoonup f$ in $L^{2}(\Omega_{T})$. Affinity of $f(x,t,\cdot)$ on $(0,\lambda)$ for small $\lambda$ and the compactness of $s_{\delta}$ on $(\lambda,1]$ results in $f_{\delta}\rightarrow f$ in $L^{2}(\Omega_{T})$. Using again Remark~\ref{s_01}, the assumption $k(0)=0$ and the compactness of $s_{\delta}$ on $(\lambda,1]$ provides $k_{\delta}(s_{\delta}) \rightarrow k(s)$ in $L^{2}(\Omega_{T})$ and $\nabla k_{\delta}(s_{\delta}) \rightharpoonup \nabla k(s)$ in $(L^{2}(\Omega_{T}))^3$ for $\delta \rightarrow 0$. Therefore, we can identify \eqref{bounda}$_{1}$ as the distributional limit of \eqref{regsys}$_{1}$.

\paragraph*{Compactness of the families $s_{\delta}$ and $u_{\delta}$.} 
We outline the ideas of the proof and refer to~\citep{Schweizer07} for more details. As $\Phi_{\delta}^{\prime}$ is unbounded, we only aim at compactness for $u_{\delta}$ away from regions with maximal saturation. To prove compactness  for the familiy $u_{\delta}$, we use a sequence $\varepsilon \rightarrow 0$ and an associated sequence of cut-off functions $\alpha_{\varepsilon} \in C^{\infty}(\Omega,[0,1])$ with $\alpha_{\varepsilon}(x) = 1$ for all $x \in \Omega$ with $\mathrm{dist}(x,\partial \Omega) \geq \varepsilon$. If we consider  $\eta_{\varepsilon}(\xi)\colonequals (\xi + \varepsilon)_{-}$ for fixed $\varepsilon$, it can be shown that the familiy $\alpha_{\varepsilon} \cdot \eta_{\varepsilon}(u_{\delta})$ is compact in $L^{2}(\Omega_{T-\varepsilon})$ and that for a subsequence $\delta\rightarrow 0$ there holds 
\begin{equation}\label{cutoff}
\eta_{\varepsilon}(u_{\delta}) \overset{\delta \rightarrow 0}{\longrightarrow} \eta_{\varepsilon}(u) \enspace \mbox{in} \enspace L^{2}(\Omega_{T}) \enspace \mbox{for any} \enspace \varepsilon > 0.
\end{equation}
Here, the compactness of the familiy $\alpha_{\varepsilon} \cdot \eta_{\varepsilon}(u_{\delta})$ in $L^{2}(\Omega_{T-\varepsilon})$ can be deduced using the Riesz characterization of compact sets. The compactness of the familiy~$s_{\delta}$ away from regions with zero saturation follows similarly.


For the proof of relation \eqref{bounda}$_{2}$, the boundary conditions on the Neumann and the outflow boundary and the initial conditions we refer to~\cite{Schweizer07}. 
We only remark that \eqref{bounda}$_{2-4}$, \eqref{bounda}$_{6}$ and \eqref{bounda}$_{8}$ follow more or less straightforward, while the proof of \eqref{bounda}$_{5}$ is much more challenging and uses defect measures and a compensated compactness argument.
We now turn to the proof of \eqref{bounda}$_7$. The problem we have to cope with on the infiltration boundary is the fact that 
$h(u)$ is not Lebesgue-integrable for $S := \{(x,t) \in \Omega_{T}\, |\, s(x,t) = 0\}$ not being a null set in the sense of Lebesgue. Since we are not able to bound $h_{\delta}(u_{\delta})$ on $S$ independently of $\delta$, $h_{\delta}(u_{\delta})$ is not necessarily Lebesgue-integrable on $S$ either. To prove convergence on the infiltration boundary, we exploit that $p_{-}(x,t) = -h(u(x,t))$  and that the water pressure $p(x,t)$ equals the capillary pressure $p_{c}(s)$. Thanks to Assumption~\ref{ass1}, we know that $(-p)_{+}$ is bounded and therefore that $h(u)$ is bounded, too.

Before analyzing the limiting process for $\delta\to 0$, we show that our boundary condition on $\Sigma_\text{in}$ 
\begin{equation*}
(\nabla u_{\delta} + k_{\delta}(s_{\delta}) \cdot {\boldsymbol e}) \cdot {\boldsymbol \nu} = - \frac{(u_{\delta})_{+}}{k_{\delta}(1)c} + \frac{w_{\delta}}{c}+ \frac{h_{\delta}(u_{\delta})\psi(w_{\delta})}{c}
\end{equation*}
is well-defined in the sense of traces. Since $w_{\delta}$ is only defined on $\Sigma_\text{in}$, we do not have to discuss the terms soley depending on $w_{\delta}$. Due to Theorem \ref{apri} c), the function $(u_{\delta})_{+}$ has a well-defined trace on $\Sigma_{\text{in},T}$, which leaves us the term $h_{\delta}(u_{\delta})$ and, particularly, because of Theorem \ref{apri} e),
\begin{equation*}
h_{\delta}^{\prime}(u_{\delta}) = \frac{(1 - \mbox{sign}(u_{\delta}))}{2\left(k_{\delta}\left(\Phi_{\delta}^{-1}\left(1/2(u_{\delta} - |u_{\delta}|)\right)\right)\right) } .
\end{equation*}
For simplicity, we restrict our analysis to the areas where $u_{\delta} \leq 0$. This is not
a constraint, because $h_{\delta}^{\prime}(u_{\delta})$ is bounded in the complement of these areas in~$\Omega_{T}$ anyway. We arrive at $h_{\delta}^{\prime}(u_{\delta}) \sim 1/\left(k_{\delta}\left(\Phi_{\delta}^{-1}(u_{\delta})\right)\right).$ Obviously, $h_{\delta}^{\prime}(u_{\delta})$ is unbounded if $s_{\delta} = \Phi_{\delta}^{-1}(u_{\delta}) \rightarrow 0.$ As already discussed at the beginning of this paragraph, this would be a contradiction to the fact that $(-p)_{+}$ is bounded due to assumption~\eqref{crucial_impli}. This proves that $h_{\delta}^{\prime}(u_{\delta})$ is in fact bounded and, therefore, that $h_{\delta}(u_{\delta})$ has a well-defined trace.

Since $\nabla h_{\delta}(\eta_{\varepsilon}(u_{\delta}))$ is bounded in $(L^{2}(\Omega_{T}))^3$, we get a subsequence, converging weakly to a limit $\tilde h$. Using \eqref{cutoff} and Assumption \ref{ass1}, one can than easily prove that $\tilde h = \nabla h(\eta_{\varepsilon}(u))$ and $h_{\delta}(\eta_{\varepsilon}(u_{\delta})) \rightarrow h(\eta_{\varepsilon}(u))$ in $L^{2}(\Omega_{T})$. If we now account for the fact that $h_{\delta}(u_{\delta})$ only depends on the negative part of $u_{\delta}$, we get that $\nabla h_{\delta}(u_{\delta}) \rightharpoonup \nabla h(u)$ and $h_{\delta}(u_{\delta}) \rightarrow h(u)$ in $(L^{2}(\Omega_{T}))^3$ and $L^{2}(\Omega_{T})$, respectively, for $\delta \rightarrow 0$. This results in $h_{\delta}(u_{\delta}) \rightarrow h(u)$ in $L^{2}(\Sigma_{\text{in},T})$. The convergence $u_{\delta} \rightarrow u$ in $L^{2}(\Sigma_{\text{in}, T})$ implies directly the convergence $(u_{\delta})_{+} \rightarrow u_{+}$ in $L^{2}(\Sigma_{\text{in}, T})$ for $\delta \rightarrow 0.$ 
Applying these results and Gronwall's Lemma, one gets $w_{\delta} \rightarrow w$ in $C^{0}([0,T], L^{1}(\Sigma_\text{in}))$,  implying that $w_{x}(0) = w_{0}$ and $w \geq 0$ a.e.\ on $\Sigma_\text{in}$ for all $t \in [0,T]$.
\end{proof}

\section{Discretization and convex minimization}
\label{discreteConvex}

In this section we introduce an implicit--explicit time discretization of System~\eqref{system}, followed by a space
discretization with finite elements for the Richards equation. We are bound to this particular discretization in the sense
that it allows for spatial convex minimization problems in the porous medium that on the discrete level can be 
treated by monotone multigrid. We refer to \cite{BerningerKornhuberSander11} for a detailed analytical and numerical study of this approach. Here, we extend the approach in \cite{BerningerKornhuberSander11}, replacing Dirichlet conditions by nonlinear physical Robin conditions. For the treatment of Robin conditions, in which the physical pressure appears without nonlinear deformation, see also~\cite{BerningerSander10}. 
The ODE's for the surface water are treated explicitly. For another example of coupling Richards' equation with a single ODE representing the water height of a lake given by a compartment model, we refer to \cite[Sec.\;4.3]{BerningerDiss}.
See \cite{Bastian_et_al:2012} and \cite{BernKornSand11} for implicit treatments of the coupling of Richards' equation with the shallow water equations, where in the latter article---as in this paper---the coupling is also given by a Robin transmission condition.

\subsection{Time discretization}

We introduce our special implicit--explicit time discretization of System~\eqref{system}, first in a strong and then
in a weak formulation. Special attention will be given to the nonlinear Robin condition on $\Sigma_\text{in}$.
The Signorini-type condition will lead to a variational inequality.


\subsubsection{Strong formulation}

Given $s^n$ in $\Omega$ and $w^n$ as well as $r^{n+1}$ on $\Sigma_\text{in}$ for time steps $t_n$ and $t_{n+1}$, respectively, with $\tau=t_{n+1}-t_n$, we define $(s^{n+1},w^{n+1})$ at time step $t_{n+1}$ as follows.
\begin{alignat}{2}\label{system_time_strong}
u^{n+1} \in \Phi(s^{n+1}),\quad \nonumber s^{n+1}  = & \tau\divergence (\nabla u^{n+1} + k(s^n) \cdot {\boldsymbol e}) + \tau f(s^{n+1})+s^n &\quad&\mbox{in} \enspace \Omega,\\
\nonumber u^{n+1} \leq 0 \enspace \mbox{ and} &\enspace -(\nabla u^{n+1} + k(s^n) \cdot {\boldsymbol e})\cdot {\boldsymbol \nu} \geq 0 &&\hspace*{-.7cm}\mbox{on} \enspace \Sigma_\text{out},\\
0 &= ((\nabla u^{n+1} + k(s^n) \cdot {\boldsymbol e})\cdot {\boldsymbol \nu}) \cdot u^{n+1} &&\hspace*{-.7cm}\mbox{on} \enspace \Sigma_\text{out},\\
\nonumber (\nabla u^{n+1} + k(s^n) \cdot {\boldsymbol e})\cdot {\boldsymbol \nu} &= 0 &&\hspace*{-.7cm}\mbox{on} \enspace \Sigma_{N},\\
\nonumber (\nabla u^{n+1} + k(s^n) \cdot {\boldsymbol e})\cdot {\boldsymbol \nu} &= g(u^{n+1},w^n) &&\hspace*{-.7cm}\mbox{on} \enspace \Sigma_\text{in},\\
\nonumber w^{n+1} &= \tau(r^{n+1} - g(u^{n+1},w^{n}))+w^n &&\hspace*{-.7cm}\mbox{in} \enspace \Sigma_\text{in}.
\end{alignat}

Here, boundary condition \eqref{system_time_strong}$_3,$ \eqref{system_time_strong}$_4$ is the discretization of 
the original Signorini-type condition~\eqref{outflow} which is equivalent to \eqref{system}$_3$--\eqref{system}$_5$.
Boundary condition \eqref{system_time_strong}$_6$ needs special consideration. 
First, we rewrite $g$ on $\Sigma_\text{in}\times\Sigma_\text{in}$ (defined in \eqref{system}) by inserting the inverse 
Kirchhoff transformation $\kappa^{-1}:u\mapsto p$ given in Equations~\eqref{inverseKirchhoff} and obtain
\begin{align*}
-g(u^{n+1},w^{n})+ {{}} \frac{w^n}{c}
&=\frac{{{}} u^{n+1}_{+}}{k(1)c} - \frac{{{}} h(u^{n+1}) \min \{ 1, (\frac{w^n}{\sigma} )_{+}\}}{c} \\
&=c^{-1}\kappa^{-1}(u^{n+1})_++c^{-1}\kappa^{-1}(u^{n+1})_-\cdot\psi(w^n).
\end{align*}
Therefore, with $p^{n+1}\colonequals\kappa^{-1}(u^{n+1})$, boundary condition $\eqref{system_time_strong}_6$ reads
\begin{align*}
(k(s^{n+1})\nabla p^{n+1} + k(s^n) \cdot {\boldsymbol e})\cdot{\boldsymbol \nu}+c^{-1}p^{n+1}_++c^{-1}p^{n+1}_-\cdot\psi(w^n)&=\\[1mm]
(\nabla u^{n+1} + k(s^n) \cdot {\boldsymbol e})\cdot{\boldsymbol \nu}+c^{-1}\kappa^{-1}(u^{n+1})_++c^{-1}\kappa^{-1}(u^{n+1})_-\cdot\psi(w^n)
&=c^{-1}{{}} w^n
\end{align*}
on $\Sigma_\text{in}$ and is a nonlinear Robin condition both in $p^{n+1}$ and in $u^{n+1}$. The special structure of it when formulated in $u^{n+1}$ stems 
from the fact that the real functions
\begin{equation}
\label{superposition}
\kappa^*_\xi:u\mapsto c^{-1}\kappa^{-1}(u)_++c^{-1}\kappa^{-1}(u)_-\cdot\psi(w^n(\xi)),
\end{equation}
which, for $\xi\in\Sigma_\text{in}$ almost everywhere, induce a superposition operator on $\Sigma_\text{in}$, are increasing, cf.~\cite{Berninger09}. This observation is essential for our way to 
determine $u^{n+1}$ that we present below. In fact, for $w^n(\xi)\geq \sigma$, the function $c\cdot\kappa^*_\xi$ is nothing but the real function $\kappa^{-1}$ that induces
the inverse Kirchhoff transformation~\eqref{inverseKirchhoff}. Otherwise, the factor 
$\psi(w^n(\xi))=(\frac{w^n(\xi)}{\sigma})_{+}< 1$ just leads to a smaller weight
for negative values of $p$, see Fig.~\ref{fig:kappas}.
\begin{figure}
\begin{center}
\includegraphics[width=0.4\textwidth]{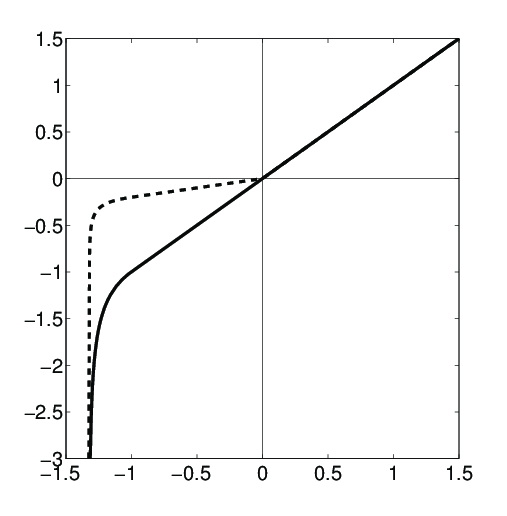}\hspace*{0.05\textwidth}
\includegraphics[width=0.4\textwidth]{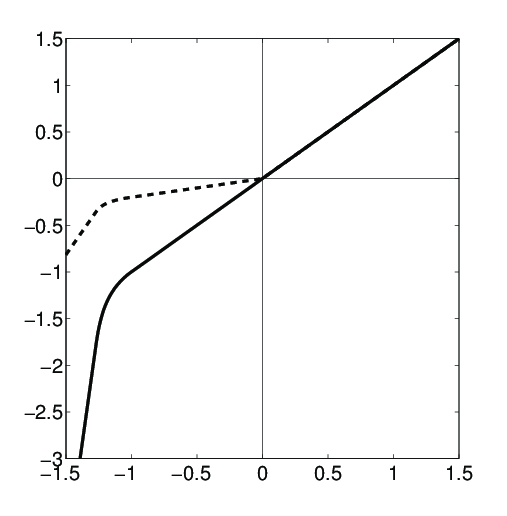}
\vspace*{-0.03\textwidth}
\end{center}
\caption{Left: degenerate $\kappa^*_\xi=\kappa^{-1}$ with $\Phi^{-1}(0)=-1.3245\cdot |p_b|>-\infty$ by Brooks--Corey, (bold line, data from Section~\ref{numerics}) and with weight $\psi(w^n(\xi))=0.2$ for $u_-$ (dotted line). Right: regularized $(\kappa^*_\xi)_\delta$ with $\delta^2=0.1$ and $\Phi^{-1}(0)=-\infty$, (bold line, construction by $k_\delta(s)\colonequals\max(k(s),\delta^2)$ as in \cite[Sec.\;1.4.3]{BerningerDiss}) and with weight $\psi(w^n(\xi))=0.2$ for $u_-$ (dotted line).}
\label{fig:kappas}
\end{figure}
Finally, note that system \eqref{system_time_strong} only contains an explicit coupling of ground and surface water 
in the sense that the solution $u^{n+1}$ in the porous medium depends implicitly on the known value $w^{n}$ from the 
previous time step while $w^{n+1}$ can be explicitly determined once $u^{n+1}$ is known.

\subsubsection{Weak formulation}


In contrast to the considerations in Section~
\ref{sec:existence}, our weak formulation of the spatial problem \eqref{system_time_strong} is based on 
the second line of~\eqref{richKirch}, concentrating on $u^{n+1}$ instead of $s^{n+1}$ as the unknown on $\Omega$.
Recall that by the Kirchhoff transformation~\eqref{Kirchhoff}, we have $u\geq \Phi(0)$. Depending on the kind 
of degeneracy of $k(p^{-1}_c(q))$ for $q\to -\infty$, we might have $\Phi(0)=-\infty$ or $-\infty<\Phi(0)<0$ as in the case of Brooks--Corey parameter functions (see \eqref{BrooksCoreysat} and \eqref{BrooksCoreykr} or \cite[Sec.\;1.3]{BerningerDiss} and Fig.~\ref{fig:kappas}). Anyway, for a weak formulation we should expect $u^{n+1}$ to be an element of the convex subset
\begin{equation*}
\mathcal{K}\colonequals\{v\in H^1(\Omega): v\geq \Phi(0),\enspace v_{|\Sigma_\text{out}}\leq 0\}
\end{equation*}
of the Sobolev space $H^1(\Omega)$. Together with the Signorini-type outflow condition~\eqref{outflow}, 
which is also encoded in $\mathcal{K}$, these constraints lead to the following weak interpretation of 
\eqref{system_time_strong}$_1$--\eqref{system_time_strong}$_6$ 
in terms of a variational inequality for $u^{n+1}\in\mathcal{K}$ (for details consult \cite[Sec.\;1.5.2/3,~3.4.5]{BerningerDiss}, where ${\mathbf e}=-e_z$)\,:
\begin{multline}\label{system_time_weak}
\int_\Omega (\Phi^{-1}(u^{n+1})-\tau f(\Phi^{-1}(u^{n+1})))\cdot (v-u^{n+1})
+ \tau\int_{\Sigma_\text{in}}\kappa^*_\xi(u^{n+1})\cdot (v-u^{n+1}) \\
+ \tau\int_\Omega\nabla u^{n+1}\cdot\nabla (v-u^{n+1})
- \int_\Omega \Phi^{-1}(u^{n})\cdot (v-u^{n+1}) \\[1mm]
+ \tau\int_\Omega k(\Phi^{-1}(u^n)){\boldsymbol e}\cdot \nabla(v-u^{n+1})
- \tau\int_{\Sigma_\text{in}}\frac{w^n}{c}\cdot (v-u^{n+1})\geq 0\qquad\forall v\in\mathcal{K}.
\end{multline}

\subsection{Convex theory}

In this section, we show that under reasonable assumptions, the variational inequality \eqref{system_time_weak}
is equivalent to a uniquely solvable convex minimization problem.

Let $\Psi_x$, $x\in\Omega$, and $\Psi_\xi$, $\xi\in\Sigma_\text{in}$, be primitives of $\Phi^{-1}-\tau f(x,t_{n+1}(\Phi^{-1}(\cdot)))$, $x\in\Omega$, and $\tau\kappa^*_\xi$, $\xi\in\Sigma_\text{in}$, respectively, 
on $[\Phi(0),\infty)\cap\R$. Since $\Phi^{-1}-\tau f\Phi^{-1}$ and $\kappa^*_\xi$ are increasing (cf.~Assumption~\ref{ass_data}),
$\Psi_x$ and $\Psi_\xi$ are convex so that the functional
\begin{multline}
\label{convex_functional}
F:v\mapsto\int_\Omega \Psi_x(v) + \int_{\Sigma_\text{in}}\Psi_\xi(v) + \frac{\tau}{2}\int_\Omega|\nabla v|^2 \\[1mm]
- \int_\Omega \Phi^{-1}(u^{n})\cdot v 
+ \tau\int_\Omega k(\Phi^{-1}(u^n)){\boldsymbol e}\cdot \nabla v
- \tau\int_{\Sigma_\text{in}}\frac{w^n}{c}\cdot v\qquad\forall v\in\mathcal{K}
\end{multline}
is strictly convex. Due to the assumptions on $k$, $p_c$ and $f$, the functions $\Psi_x$, $x\in\Omega$, are bounded and Lipschitz continuous.
Note, however, that we have $\kappa^*_\xi(u)\to -\infty$ for $u\searrow\Phi(0)$ in case of $w^n(\xi)>0$ so that as least for \mbox{$\Phi(0)>-\infty$,} 
the transformation induced by $\kappa^*_\xi$, $\xi\in\Sigma_\text{in}$, is ill-posed in general and the second integral in \eqref{system_time_weak} is not well-defined as an
$L^1$-integral, cf.\ left graphics in Fig.~\ref{fig:kappas}. However, if we regularize $k$ and $p_c$ according to Assumption~\ref{ass2}, the regularized functions $(\kappa_\xi^*)_\delta$, $\xi\in\Sigma_\text{in}$,
turn out to be Lipschitz continuous on $\R$
with Lipschitz constant $\max_u(\kappa_\xi^*)_\delta'(u)=\frac{1}{c\delta^2}\min\{1,(\frac{w^n(\xi)}{\sigma})_+\}$
and the corresponding integral in \eqref{system_time_weak} is well-defined, cf.\ right graphics in Fig.~\ref{fig:kappas}.
The following theorem states unique solvability of the variational inequality \eqref{system_time_weak} for regularized functions $(\kappa_\xi^*)_\delta$, $\xi\in\Sigma_\text{in}$, if their inverses have uniformly bounded Lipschitz constants for $\xi\in\Sigma_\text{in}$ or, equivalently, if the weighting factor $\psi(w^n(\xi))$ is uniformly bounded away from zero on a submanifold of $\Sigma_\text{in}$ with positive Hausdorff measure. One can even prove well-posedness of~\eqref{system_time_weak}, cf.~\cite[Prop.~2.4.11]{BerningerDiss}.

\begin{theorem}
\label{theorem_convex_mini}
Let $\Psi_x$, $x\in\Omega$, be induced by coefficient functions according to Assumptions~\ref{ass1} or~\ref{ass2}. Let $\Psi_\xi$, $\xi\in\Sigma_\text{in}$, be 
induced by regularized coefficient functions according to Assumptions~\ref{ass2} and, accordingly, let $\kappa^*_\xi$, $\xi\in\Sigma_\text{in}$, be replaced by 
$(\kappa^*_\xi)_\delta$ in \eqref{system_time_weak}. Furthermore, let $w^n\in L^2(\Sigma_\text{in})$. Then the variational inequality 
\eqref{system_time_weak} is equivalent to the convex minimization problem
\begin{equation}\label{convex_mini}
u^{n+1}\in\mathcal{K}:\quad
F(u^{n+1})\leq F(v)\qquad\forall v\in\mathcal{K}.
\end{equation}
If $w^n>0$ holds on a nonempty submanifold of $\Sigma_\text{in}$, then \eqref{convex_mini} is uniquely solvable.
\end{theorem}

\begin{proof}
Both the equivalence proof and the proof for the unique solvability is a combination of the convex theory on
the time-discretized Richards equation with Signorini-type boundary conditions presented in \cite[Sec.\;2.3]{BerningerDiss} and \cite{BerningerKornhuberSander11}
and with Robin conditions which can be found in \cite[Sec.\;3.4.1]{BerningerDiss} and \cite{BerningerSander10}. Here, Lipschitz continuity of
the regularized $(\kappa^*_\xi)_\delta$ is needed.
The additional dependency of $\Psi_x$ on $x\in\Omega$ and $\Psi_\xi$ on $\xi\in\Sigma_\text{in}$ can be easily included in the theory, taking into account $f(\cdot,t_{n+1},s^{n+1})\in L^\infty(\Omega)$ and $\psi(w^n)\in L^2(\Sigma_\text{in})$ which, in particular, leads to well-definedness of all considered integrals.

The only issue we need to consider here is the lack of a Dirichlet boundary with positive Hausdorff measure
which guarantees coercivity of the 
functional~$F$. In our case, the latter is provided by the nonlinear Robin boundary condition with the help of Poincar\'e's 
inequality \cite[p.\;127]{DautrayLions2} in $H^1(\Omega)$ which gives
\begin{equation}
\label{Poincare}
\biggl\|v-\frac{1}{|\Omega|}\int_\Omega v\biggr\|_{L^2(\Omega)}\leq c_\Omega \|\nabla v\|_{L^2(\Omega)}\quad\forall v\in H^1(\Omega)
\end{equation}
for some constant $c_\Omega>0$ depending on $\Omega$. In detail, for the
nonempty submanifold $\Sigma_{w}$ of $\Sigma_\text{in}$ 
where $w^n>0$ holds, i.e., $w^n\geq w_0$ for a $w_0>0$, we prove that
\begin{equation}
\label{equivnorm}
\|v\|_{\Sigma_{w}}\colonequals\|\nabla v\|_{L^2(\Omega)}+\|v\|_{L^2(\Sigma_{w})}\quad\forall v\in H^1(\Omega)
\end{equation}
defines a norm on $H^1(\Omega)$ which is equivalent to $\|\,{\cdot}\,\|_{H^1(\Omega)}$. To see this, note that by the 
trace theorem \cite[p.\;1.61]{BrezziGilardi87}, there is a $c_{\Sigma_{w}}>0$ such that the estimate
\begin{equation*}
\|v\|_{\Sigma_{w}}\leq c_{\Sigma_{w}}\|v\|_{H^1(\Omega)}\quad\forall v\in H^1(\Omega)
\end{equation*}
holds for the seminorm $\|\,{\cdot}\,\|_{\Sigma_{w}}$. Moreover, if $\|v\|_{\Sigma_{w}}=0$, then we have 
\mbox{$v\equiv c_w\in \R$} on $\Omega$ by \eqref{Poincare} and, furthermore, $v\equiv c_w$ on $\Sigma_{w}$ by 
the trace theorem, i.e., $c_w=0$. Therefore, $\|\,{\cdot}\,\|_{\Sigma_{w}}$ is a norm on $H^1(\Omega)$ 
equivalent to $\|\,{\cdot}\,\|_{H^1(\Omega)}$, cf.~\cite[p.\;153]{Werner05}.

W.l.o.g., we assume $\Psi_\xi(0)=0$ for all $\xi\in\Sigma_\text{in}$ and finish the proof by showing that there is a $c_\delta>0$ such that
\begin{equation*}
c_\delta\|v\|_{L^2(\Sigma_{w})}^2\leq\int_{\Sigma_\text{in}}\Psi_\xi(v)\quad\forall v\in H^1(\Omega).
\end{equation*}
We prove that $c_\delta v^2\leq \Psi_\xi(v)$ for all $\xi\in\Sigma_{w}$. For all $\xi\in\Sigma_\text{in}$, $\Psi_\xi(0)$ is the minimum of 
$\Psi_\xi$ since $\Psi_\xi'(0)=\tau(\kappa^*_\xi)_\delta(0)=0$ by \eqref{inverseKirchhoff} and \eqref{superposition}. Therefore, it remains to show that 
$2c_\delta v\leq \tau(\kappa^*_\xi)_\delta(v)$ for $v\in\R$ and all $\xi\in\Sigma_{w}$ or, equivalently, that $\tau(\kappa^*_\xi)_\delta'\geq 2c_\delta$ for $\xi\in\Sigma_{w}$. 
By \eqref{inverseKirchhoff} and \eqref{superposition} we can choose 
$c_\delta\leq \frac{\tau}{2c\, k_\delta(1)}\min\{1,\frac{w_0}{\sigma}\}$.
\end{proof}

\subsection{Finite element discretization and monotone multigrid} 

In the following, we shortly indicate our space discretization and numerical solution technique for~\eqref{system_time_weak}. 
For more details we refer the reader to \cite[Sec.\;2.5,~3.4.5]{BerningerDiss}. We discretize \eqref{system_time_weak} 
or \eqref{convex_functional}, respectively, by piecewise linear finite elements. 
For simplicity, we consider the two-dimensional case of a polygonal domain $\Omega\subset\R^2$. 
Let $\mathcal{T}_j$, $j\in\N_0$, be a conforming triangulation of $\Omega$. 
The set of all vertices of the triangles in $\mathcal{T}_j$ is denoted by $\mathcal{N}_j$, 
the set of those vertices lying on $\overline\Sigma_\text{in}$ and $\overline\Sigma_\text{out}$ shall be called $\mathcal{N}_{j}^\text{{in}}$ and 
$\mathcal{N}_{j}^\text{{out}}$, respectively. The triangulation shall resolve the parts of the boundary corresponding 
to different boundary conditions.

Let $\mathcal{S}_j\subset H^1(\Omega)$ be the finite element space of all continuous functions in $H^1(\Omega)$ 
which are affine on each triangle $t\in\mathcal{T}_j$.
To discretize the surface water level $w$ we consider the trace grid $\mathcal{T}_j^\text{{in}}$ of $\mathcal{T}_j$
on $\overline\Sigma_\text{in}$, and its dual grid~$\mathcal{T}_j^{\text{{in}},*}$.  On $\mathcal{T}_j^{\text{{in}},*}$
we define the finite element space $\mathcal{S}_j^\text{{in}}$ as the space of all functions that are constant
on each element.  This is a suitable space for $w$, as the evolution equation~\eqref{ode1}
does not involve space derivatives.  Note that each element of $\mathcal{T}_j^{\text{{in}},*}$ corresponds
uniquely to a vertex $q$ of $\mathcal{N}_j^{\text{{in}}}$ which we call element center.   We will therefore sometimes use vertices of
$\mathcal{N}_j^{\text{{in}}}$ to label elements in $\mathcal{T}_j^{\text{{in}},*}$.
The nodal basis function corresponding to a node $q\in\mathcal{N}_{j}$ 
shall be called $\lambda_q^{(j)}$, and the (piecewise constant) basis function corresponding
to an element $q$ of $\mathcal{T}_j^{\text{{in}},*}$ will be called~$\lambda_q^{\text{{in}},(j)}$.

The finite dimensional analogue of $\mathcal{K}$ is the convex set
\begin{equation*}
\mathcal{K}_j\colonequals
\big\{v\in \mathcal{S}_j:v(q)\geq \Phi(0)\enspace\forall q\in\mathcal{N}_j\,\wedge\, v(q)\leq 0\enspace\forall q\in\mathcal{N}_j^{\text{out}}\big\}.
\end{equation*}
We define the $\mathcal{S}_j$-interpolation $I_{\mathcal{S}_j}:C(\overline\Omega)\to\mathcal{S}_j$ 
by evaluation at the vertices,
and the $\mathcal{S}_j^{\text{{in}}}$-interpolation $I_{\mathcal{S}_j^{\text{{in}}}}:C(\overline\Sigma_\text{in})\to\mathcal{S}_j^{\text{{in}}}$
by evaluation at the element centers of~$\mathcal{T}_j^{\text{{in}},*}$.  This allows to
define the finite element discretization of $F$ in~\eqref{convex_functional} by
\begin{multline}
\label{Fj}
F_j:v\mapsto\int_\Omega I_{\mathcal{S}_j}\Psi_x(v) + \int_{\Sigma_\text{in}}I_{\mathcal{S}_j^{\text{{in}}}}\Psi_\xi(v) + \frac{\tau}{2}\int_\Omega|\nabla v|^2 \\[1mm]
- \int_\Omega \Phi^{-1}(u^{n})\cdot v 
+ \tau\int_\Omega k(\Phi^{-1}(u^n)){\boldsymbol e}\cdot \nabla v 
- \tau\int_{\Sigma_\text{in}}\frac{w^n_j}{c}
\cdot v\qquad\forall v\in\mathcal{K}_j,
\end{multline}
where we assume that a discrete surface water level $w_j^n\in\mathcal{S}_j^{\text{{in}}}$ is available and also used in
the definition of $(\kappa^*_\xi)_\delta$ and, equivalently, of $\Psi_\xi$. The discrete finite element solution $u^n_j\in\mathcal{S}_j$ is assumed to be known, too.
Then, with the weights 
$$
h_q\colonequals\int_\Omega\lambda_q^{(j)}\quad\mbox{and}\quad h_q^{\text{{in}}}\colonequals\int_{\Sigma_\text{in}}\lambda_q^{\text{{in}},(j)},
$$
we can note the following.
\begin{theorem}
\label{discequiunique}
Assume that the conditions of Theorem \ref{theorem_convex_mini} are satisfied. Then the discrete variational inequality
for $u_j^{n+1}\in\mathcal{K}_j:$
\begin{multline}
\label{Varia3discr}
\sum_{q\in\mathcal{N}_j}\Big[\Phi^{-1}(u_j^{n+1}(q))-\tau f\big(\Phi^{-1}(u_j^{n+1}(q))\big)\Big]\,\big(v(q)-u_j^{n+1}(q)\big)\,h_q\\[1mm] 
+ \tau\sum_{q\in\mathcal{N}_{j}^{\text{{in}}}}(\kappa^*_\xi)_\delta(u_j^{n+1}(q))\,(v(q)-u_j^{n+1}(q))\,h_{q}^{\text{{in}}} \\[1mm]
+ \tau\int_\Omega\nabla u_j^{n+1}\cdot\nabla (v-u_j^{n+1})
- \int_\Omega \Phi^{-1}(u_j^{n})\cdot (v-u_j^{n+1}) \\[1mm]
+ \tau\int_\Omega k(\Phi^{-1}(u_j^n)){\boldsymbol e}\cdot \nabla(v-u_j^{n+1}) 
- \tau\int_{\Sigma_\text{in}}\frac{w^n_j}{c}\cdot (v-u_j^{n+1})
\,\geq\, 0\qquad\forall v\in\mathcal{K}_j
\end{multline}
is equivalent to the discrete convex minimization problem
\begin{equation}
\label{F_ju}
u_j^{n+1}\in\mathcal{K}_j:\quad
F_j(u_j^{n+1})\leq F_j(v)\qquad\forall v\in\mathcal{K}_j.
\end{equation}
If $w^n(q)>0$ holds for a $q\in\mathcal{N}^{\text{{in}}}_j$, then \eqref{F_ju} is uniquely solvable.
\end{theorem}

The proof of this theorem for $\mathcal{K}_j\subset\mathcal{S}_j$ is analogous to the one of its continuous counterpart,
Theorem~\ref{theorem_convex_mini}, for $\mathcal{K}\subset H^1(\Omega)$. For this 
we refer again to \cite[Sec.\;2.5]{BerningerDiss} and~\cite{BerningerSander10}. We only point out that the second integral in the definition~\eqref{Fj} of $F_j$ is the same as if $\mathcal{S}_j^{\text{{in}}}$ is replaced by the space of piecewise linear finite elements on~$\mathcal{T}_j^{\text{{in}}}$. As for~\eqref{system_time_weak}, well-posedness can also be established for~\eqref{Varia3discr}.
Under reasonable assumptions, one also obtains a convergence result for the discretization.

\begin{theorem}
Let the conditions of Theorem~\ref{theorem_convex_mini} be satisfied and $\mathcal{T}_j$, $j\geq 0$, be shape regular and
$h_j = \max_{t\in\mathcal{T}_j} \diam t \to 0$ for $j\to\infty$. Let $w_j^n\in L^2(\Sigma_\text{in})$, $j\geq 0$, be given with 
$w_j^n\to w^n$ in $L^2(\Sigma_\text{in})$ for $j\to\infty$.
Furthermore, we assume 
$C^\infty(\overline\Omega)\cap\mathcal{K}$ to be dense in $\mathcal{K}$.
We define $s^{n+1}=\Phi^{-1}(u^{n+1})$ and, with $\kappa_\delta$ obtained by regularization of $\kappa$ 
according Assumption~\ref{ass2}, we define 
$p^{n+1}=\kappa_\delta^{-1}(u^{n+1})$ (analogously, $s_j^{n+1}$ and $p^{n+1}_j$ are defined).
Then, for $j\to\infty$, we have the convergence
\begin{equation}
\label{ujconvergence}
u_j^{n+1}\to u^{n+1}\quad\mbox{in}\enspace H^1(\Omega),
\end{equation}
\begin{equation}
\label{pjconvergence}
p_j^{n+1}\to p^{n+1}\quad\mbox{in}\enspace H^1(\Omega)\quad\mbox{and}\quad I_{\mathcal{S}_j}p_j^{n+1}\to p^{n+1}\quad\text{in}\enspace L^2(\Omega).
\end{equation}
In addition, we have
\begin{equation}
\label{sjconvergence}
s_j^{n+1}\to s^{n+1}\quad\mbox{in}\enspace H^1(\Omega)\quad\mbox{and}\quad I_{\mathcal{S}_j}s_j^{n+1}\to s^{n+1}\quad\text{in}\enspace L^2(\Omega)
\end{equation}
if $\Phi^{-1}:(\Phi(0),\infty)\to\R$ is Lipschitz continuous for the first and H\"older continuous for the second statement, respectively.
\end{theorem}

The proof of \eqref{ujconvergence} is an adaptation of the one given in \cite[Thm.\;2.5.9]{BerningerDiss}, where we need to use the norm
defined in \eqref{equivnorm} instead of the $H^1$-seminorm as an equivalent norm in the solution space, which is $H^1(\Omega)$ here.
For the proofs of \eqref{pjconvergence} and \eqref{sjconvergence} see \cite{BerningerKornhuberSander11}.

In terms of finite elements, a straightforward discretization of \eqref{system_time_strong}$_7$ reads
$$
w^{n+1}_j = \tau\left(r^{n+1}_j - 
I_{\mathcal{S}^{\text{{in}}}_j}\,g\left(u^{n+1}_{j|\mathcal{S}^{\text{{in}}}_j}, w^{n}_j\right)\right)+ w^n_j 
\quad\mbox{in}\enspace \mathcal{S}^{\text{{in}}}_j.
$$
Here, we can assume that $u_j^{n+1}$ has been obtained by \eqref{Varia3discr} and that the iterates 
$r^{n+1}_j,\, w^n_j\in\mathcal{S}^{\text{{in}}}_j$ are given.

Quadrature of the second integral in \eqref{Varia3discr} is given by 
$\mathcal{S}_j$-interpolation of $\Phi^{-1}(u_j^{n})$.
In addition, with regard to the numerical treatment of the discretized problem \eqref{Varia3discr}, it is necessary
to use an upwind technique for the gravitational (convective) term $k(\Phi^{-1}(u^n)){\boldsymbol e}$ in the third integral. 
In the finite element context, this can be achieved by adding an artificial viscosity term
to the discretized convection, cf.~\cite[Sec.\;4.2]{BerningerDiss} and~\cite{BerningerKornhuberSander11}.

The discrete convex obstacle problem arising from \eqref{Varia3discr} is defined on a set~$\mathcal{K}_j$
which is a product of intervals. To such problems,
monotone multigrid methods can be applied \cite{Kornhuber02}. In short, these methods first 
minimize $F_j$ by successive one-dimensional
minimization along the $\lambda_q^{(j)}$-directions for $q\in\mathcal{N}_j$. After application of this nonlinear 
Gauss--Seidel fine grid smoother, a suitable coarse grid correction given by constraint 
Newton linearization in the smooth regimes of the nonlinear system given by \eqref{Varia3discr} accelerates
convergence. Linear multigrid convergence rates are obtained asymptotically. For details concerning the application of
this method to a Signorini-type and Robin problem like \eqref{Varia3discr} we refer to 
\cite[Sec.\;2.7,~3.4.5]{BerningerDiss}.

In the next section, we will apply monotone multigrid to an example with Brooks--Corey parameter functions
which do neither satisfy the conditions for the existence theorem~\ref{existence} nor for the theorems given in this section which require
Kirchhoff transformations given by regularized parameter functions. However, even though the equivalence result
of Theorem~\ref{theorem_convex_mini} does not hold if $(\kappa^*_\xi)_\delta$ is replaced by $\kappa^*_\xi$ 
arising from the Brooks--Corey parametrization (since \eqref{system_time_weak} is not well-defined), it turns out that \eqref{convex_mini} is still a well-posed convex minimization problem in this case. The same is true for the discretized problem (since \mbox{$\Phi(0)<u_j^{n+1}$,} one can even prove Theorem~\ref{discequiunique} here) so that \eqref{F_ju} is well-posed and can be solved by monotone multigrid methods~\cite[Sec.\;3.4.5]{BerningerDiss}.

\section{Numerical example}
\label{numerics}

We close the article with a numerical example involving the coupling
of surface water to ground water, and a Signorini-type outflow condition.
We conclude that our model and solution algorithm is well-suited for
practical applications.

\begin{figure}
\begin{center}
 \includegraphics[width=0.9\textwidth]{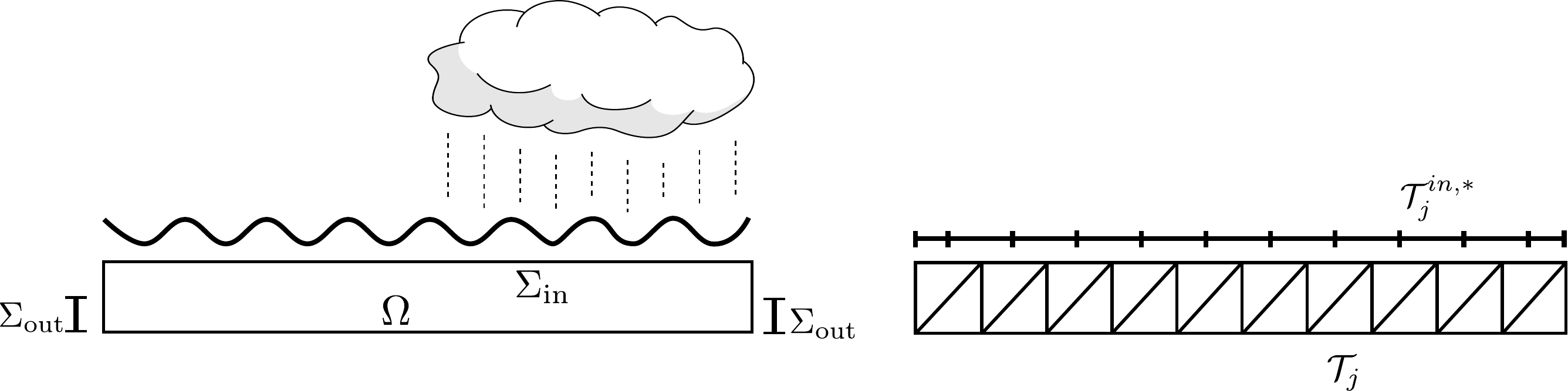}
\end{center}
\caption{Left: Setting of the numerical example.  Right: Coarse grids.  The actual grids
were created from these grids by four steps of uniform refinement.}
\label{fig:numerical_example_schematic}
\end{figure}

The problem setting is depicted schematically in the left part of Fig.~\ref{fig:numerical_example_schematic}.
As the ground water domain, we use a $10\,\text{m} \times 1\,\text{m}$ rectangle whose lower left corner shall
be situated in the origin of~$\R^2$.
Zero-flow boundary conditions are prescribed on the left,
right, and bottom sides, with the exception of two 0.5\,m stretches $\Sigma_\text{out}$ on the
lower left and right, where we set the Signorini outflow conditions~\eqref{outflow}.
The entire upper side $\Sigma_\text{in}$ couples with the surface water model
through the flux conditions~\eqref{ode1} and \eqref{q_inflow}.
The resistance $c$ is set to $10^5\,\text{s}$ (cf.~\cite{WieseNuetzmann09}),
and for the threshold parameter $\sigma$ we choose 0.02\,m.
We use sand as the porous medium, with soil parameters $n=0.437$ (porosity),
$s_m=0.0458$ (residual saturation), $s_M=1$ (maximal saturation), 
$p_b=-712.2$\,Pa (bubbling pressure), 
$\lambda=0.694$ (pore size distribution factor), and $K=6.66\cdot 10^{-9}\,\text{m}^{2}$ (absolute permeability), taken from \cite[Tables~5.3.2 and~5.5.5]{RawlsBrakensiek93}. We set $\mu = 1.002\cdot 10^{-3}\;\text{Pa\,s}$ for the dynamic viscosity of water. The parameters $p_b$ and $\lambda$ characterize the parameter functions $p\mapsto s=p_c^{-1}(p)$ and $s\mapsto k(s)$, according to Brooks--Corey~\cite{BrooksCorey64,vanGenuchten80} and Burdine~\cite{Burdine53} of the form
\begin{equation}
\label{BrooksCoreysat}
s=p_c^{-1}(p)=
\begin{cases}
s_m+(s_M-s_m)\left(\frac{p}{p_b}\right)^{-{{\lambda}}} & \text{ for $p\leq p_b$},\\
s_M & \text{ for $p\geq p_b$},
\end{cases}
\end{equation}
and
\begin{equation}
\label{BrooksCoreykr}
k(s)=K\left(\frac{s-s_m}{s_M-s_m}\right)^{3+\frac{2}{{{\lambda}}}}, \quad s\in [s_m,s_M].
\end{equation}

The domain is discretized with a uniform triangle grid on $161 \times 17$ vertices,
giving a mesh size of $h=\sqrt{2}/16$\,m.
We construct this grid by four uniform refinement steps of a uniform $10 \times 1$ grid.
As a by-product we obtain a grid hierarchy suitable for our multigrid solver.
To discretize the surface water model, we use the dual grid $\mathcal{T}_j^{\text{{in}},*}$ of the trace
of the subsurface grid on $\Sigma_\text{in}$.  The grids are depicted schematically
on the right of Fig.~\ref{fig:numerical_example_schematic}.
For the implementation we used the \textsc{Dune} libraries~\cite{bastian_et_al_II:2008}
together with the domain decomposition module \texttt{dune-grid-glue}~\cite{bastian:2010}.

%
We start with the initial pressure $-2\cdot 10^4$\,Pa in the domain, which, taking the
material parameters and~\eqref{BrooksCoreysat} into account, corresponds to the constant initial
saturation~0.1401.  The initial surface water height is set to a constant 0\,m,
and we prescribe a constant rainfall of $30$\,mm/h $\approx 8.33\cdot 10^{-6}\,\text{m}/\text{s}$ on $[5,10] \times \{1\}$, and zero
elsewhere.  The evolution is simulated up to a final time $T = 350\,000\,\text{s}$, 
which equals roughly four days.  For the time step size $\tau$ we choose 100\,s.
This is well below the upper bound on the time step size implied by the CFL condition
$$
\tau < h\frac{n\mu}{\varrho g \sup_{s\in [0,1]}|k'(s)|}=h\frac{n\mu}{K\varrho g (3+2/\lambda)},
$$
cf.~\cite[p.\;223]{BerningerDiss}, which in our setting evaluates to approximately $\tau < h\cdot1.14\cdot 10^3\,\text{s/m}$ $=200$\,s. 
For each time step we solve the spatial problem~\eqref{Varia3discr} to machine precision by the monotone
multigrid method described in~\cite{BerningerKornhuberSander11}.  We omit the solver convergence
rates and refer the interested reader to \cite{BerningerKornhuberSander11} for robustness and
efficiency studies.  While solving to machine precision usually means more work than necessary,
it eliminates the effect of solver inexactness from the following observations.

\begin{figure}
\begin{center}
\begin{minipage}{0.49\textwidth}
 \includegraphics[width=\textwidth]{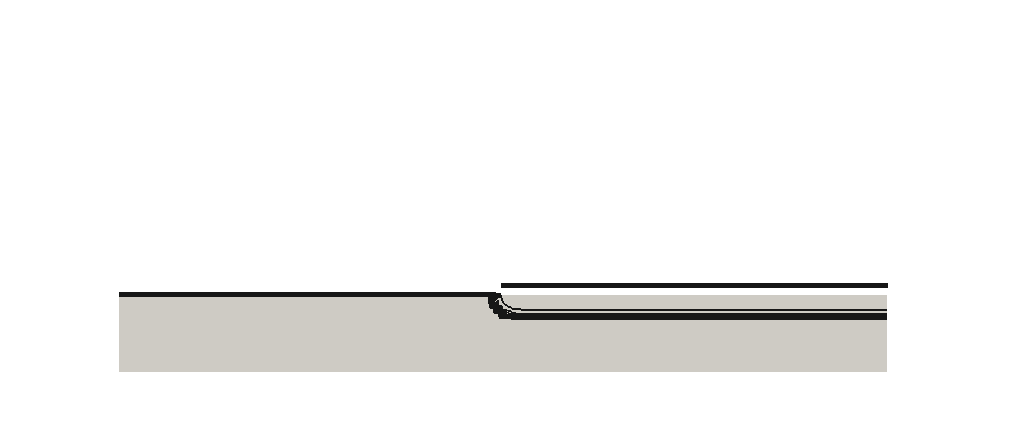}
\vspace{-2em}
 \begin{center}
   $n = 200$
 \end{center}
\end{minipage}
\begin{minipage}{0.49\textwidth}
 \includegraphics[width=\textwidth]{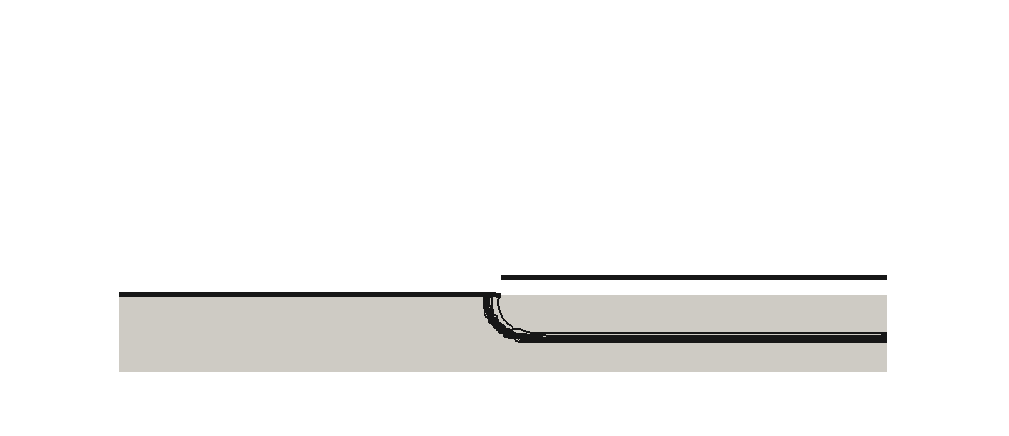}
\vspace{-2em}
 \begin{center}
   $n = 400$
 \end{center}
\end{minipage}

\begin{minipage}{0.49\textwidth}
 \includegraphics[width=\textwidth]{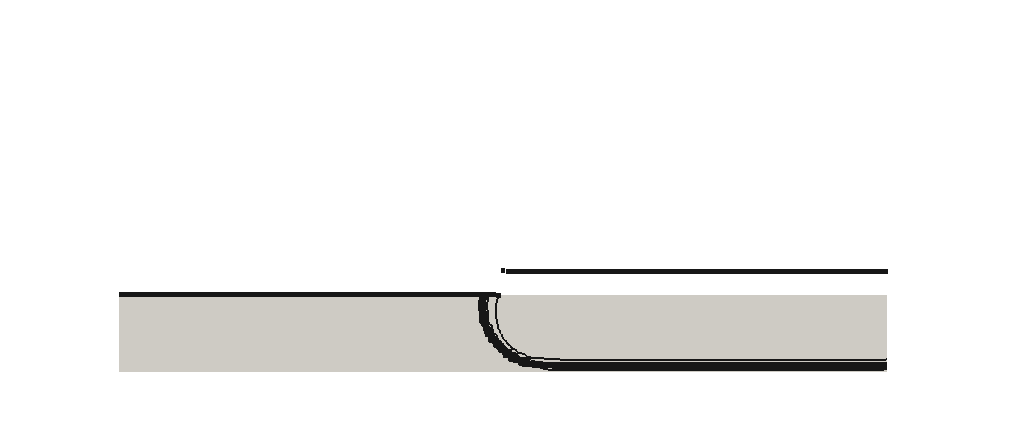}
\vspace{-2em}
 \begin{center}
   $n = 600$
 \end{center}
\end{minipage}
\begin{minipage}{0.49\textwidth}
 \includegraphics[width=\textwidth]{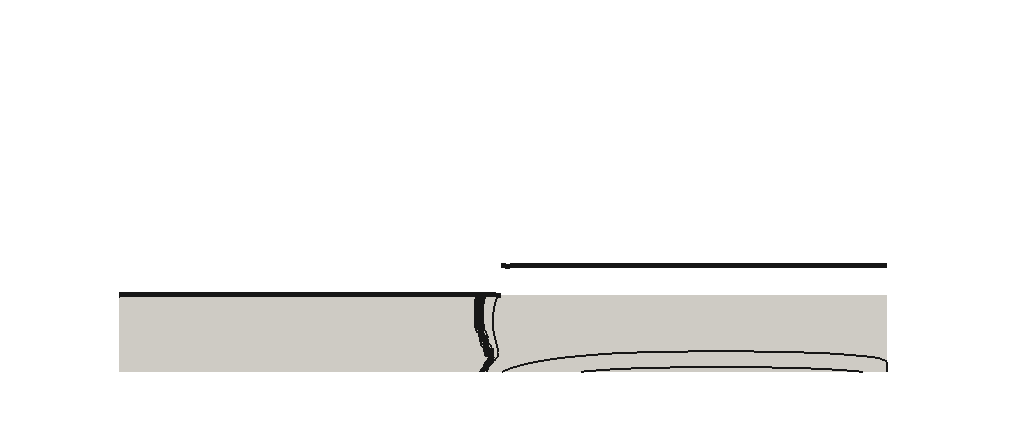}
\vspace{-2em}
 \begin{center}
   $n = 800$
 \end{center}
\end{minipage}

\begin{minipage}{0.49\textwidth}
 \includegraphics[width=\textwidth]{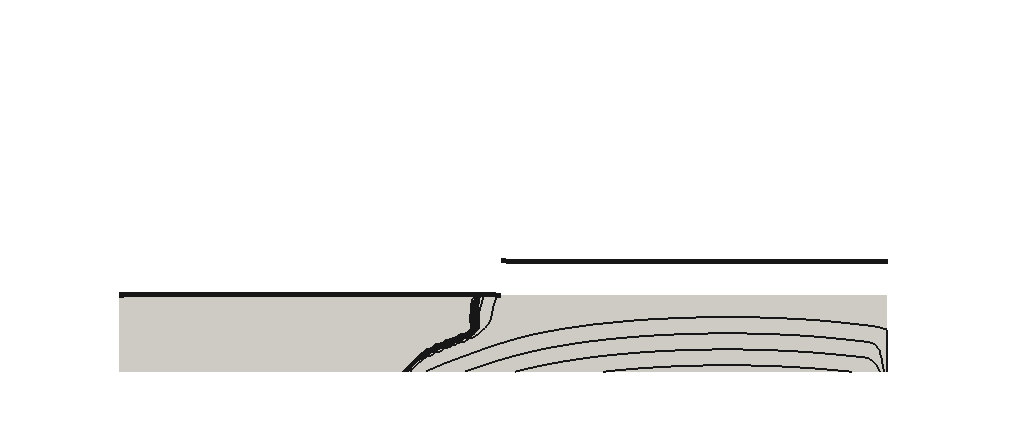}
\vspace{-2em}
 \begin{center}
   $n = 1\,000$
 \end{center}
\end{minipage}
\begin{minipage}{0.49\textwidth}
 \includegraphics[width=\textwidth]{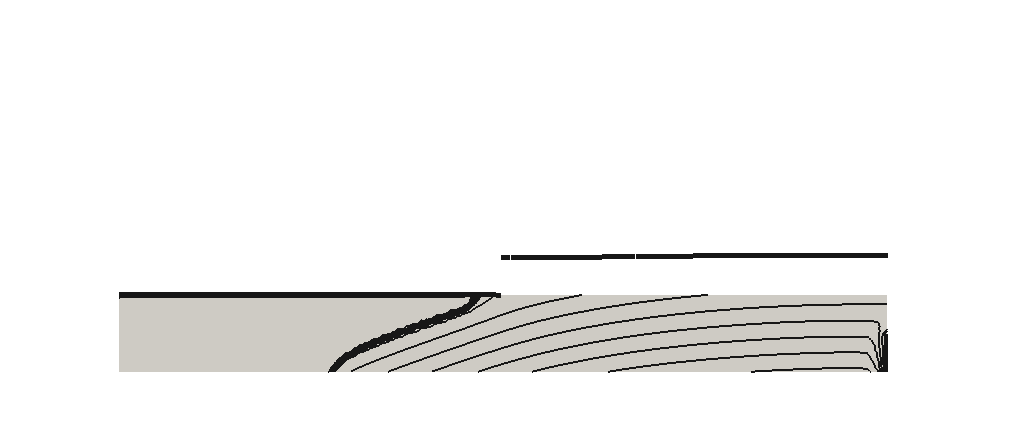}
\vspace{-2em}
 \begin{center}
   $n = 1\,200$
 \end{center}
\end{minipage}

\begin{minipage}{0.49\textwidth}
 \includegraphics[width=\textwidth]{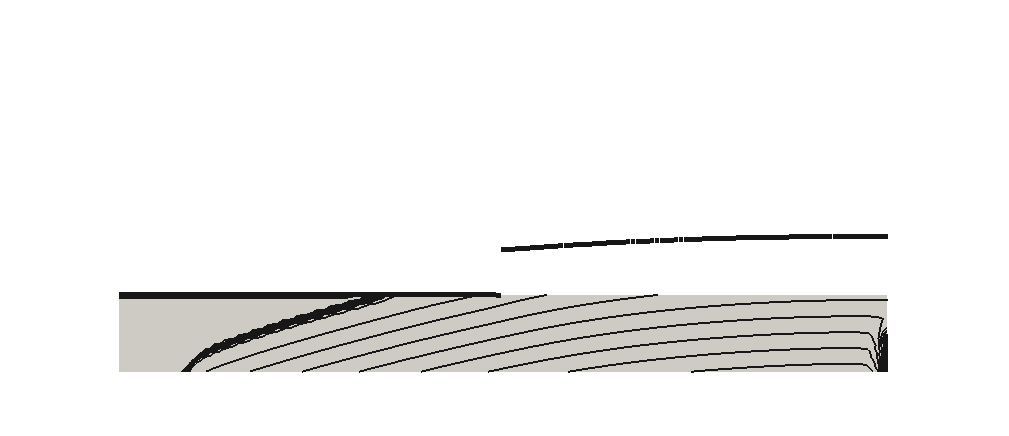}
\vspace{-2em}
 \begin{center}
   $n = 1\,600$
 \end{center}
\end{minipage}
\begin{minipage}{0.49\textwidth}
 \includegraphics[width=\textwidth]{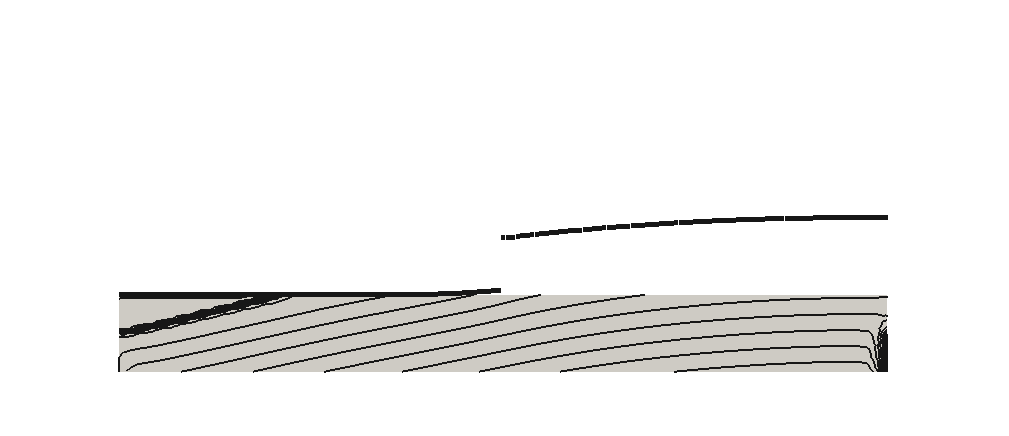}
\vspace{-2em}
 \begin{center}
   $n = 2\,000$
 \end{center}
\end{minipage}

\begin{minipage}{0.49\textwidth}
 \includegraphics[width=\textwidth]{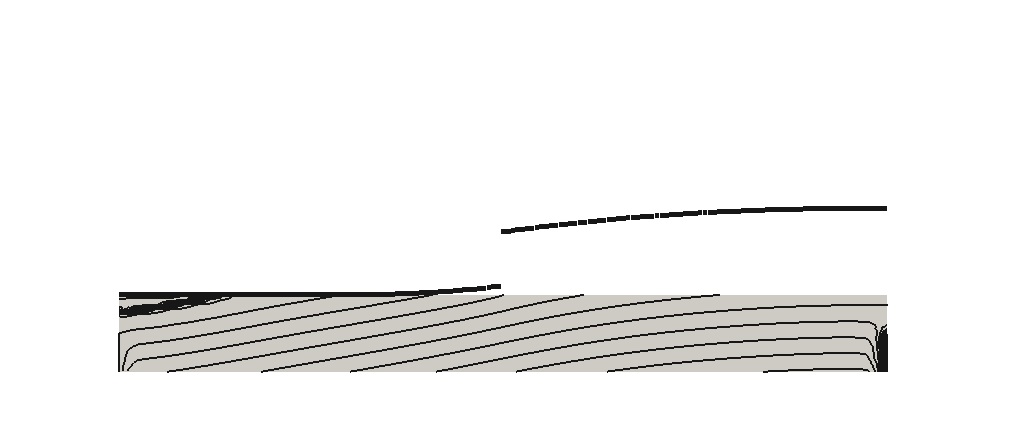}
\vspace{-2em}
 \begin{center}
   $n = 2\,200$
 \end{center}
\end{minipage}
\begin{minipage}{0.49\textwidth}
 \includegraphics[width=\textwidth]{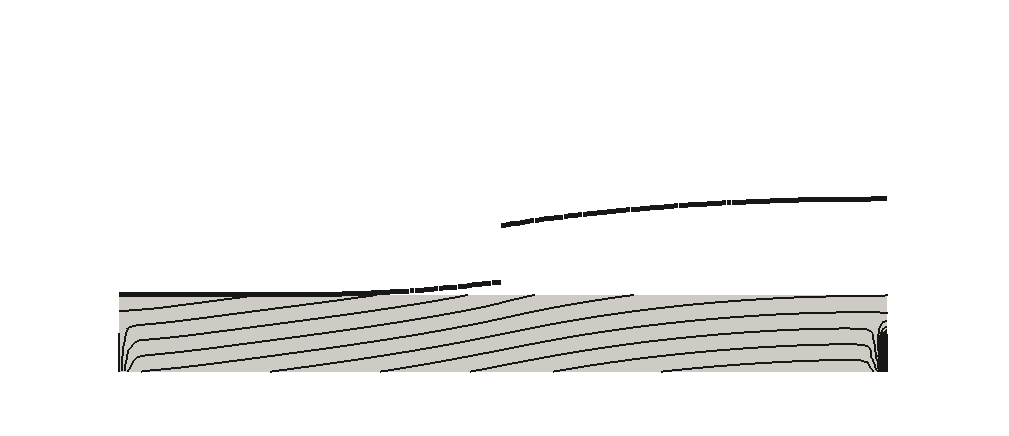}
\vspace{-2em}
 \begin{center}
   $n = 2\,400$
 \end{center}
\end{minipage}

\begin{minipage}{0.49\textwidth}
 \includegraphics[width=\textwidth]{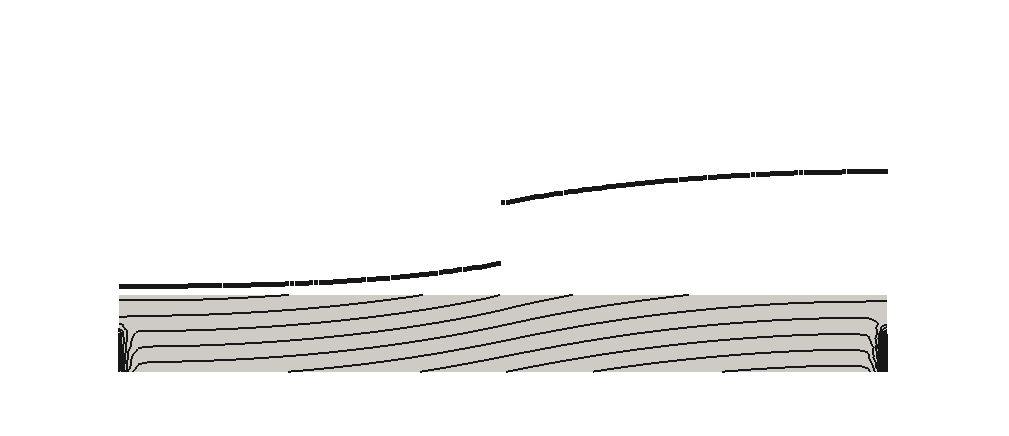}
\vspace{-2em}
 \begin{center}
   $n = 3\,000$
 \end{center}
\end{minipage}
\begin{minipage}{0.49\textwidth}
 \includegraphics[width=\textwidth]{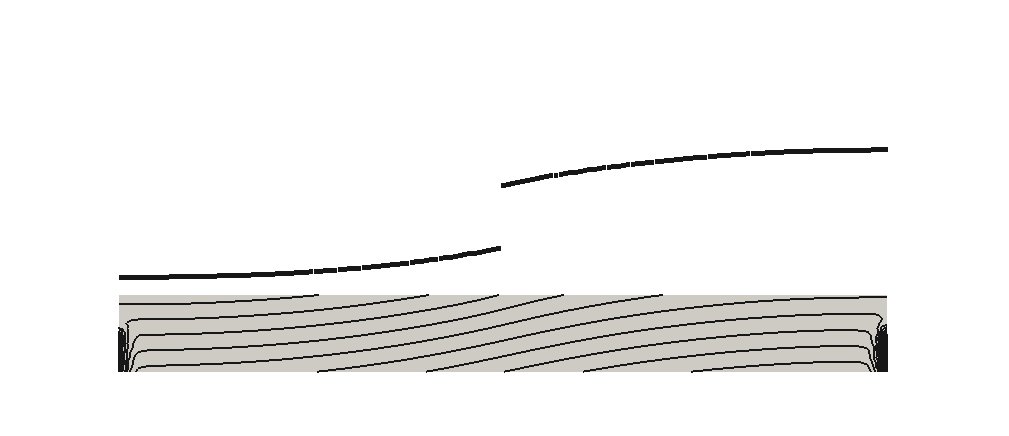}
\vspace{-2em}
 \begin{center}
   $n = 3\,500$
 \end{center}
\end{minipage}
\end{center}
 \caption{Snapshots of the numerical evolution at time steps $t_n=n\cdot\tau=n\cdot 100\,s$ as indicated, with the top black line giving the
          surface water height, and isolines of the pressure $p$
          at integral multiples \mbox{of~2\,kPa.}}
 \label{fig:solution_snapshots}
\end{figure}

Fig.~\ref{fig:solution_snapshots} shows a few snapshots from the evolution.  Initially, surface water
accumulates on the right, where there is rainfall, and infiltrates into the soil. A saturation front traverses the domain. 
When the right part of $\Omega$
is almost filled, a hydrostatic pressure distribution starts to build up there,
and the water starts to infiltrate the left part of $\Omega$.
Around time step 2\,429, the ground is fully saturated.
Most of the water entering the domain leaves it through $\Sigma_\text{out}$, while a small
part exfiltrates on the left part of $\Sigma_\text{in}$, and leads to a rising
surface water level there.  At the last time step, the subsurface pressure field is
contained in the interval $[0\,\text{Pa}, 2.53\cdot 10^4\,\text{Pa}]$, and the surface water height is contained in $[0.23\,\text{m}, 1.89\,\text{m}]$.

\bigskip

\begin{figure}
\begin{center}
 \includegraphics[width=0.7\textwidth]{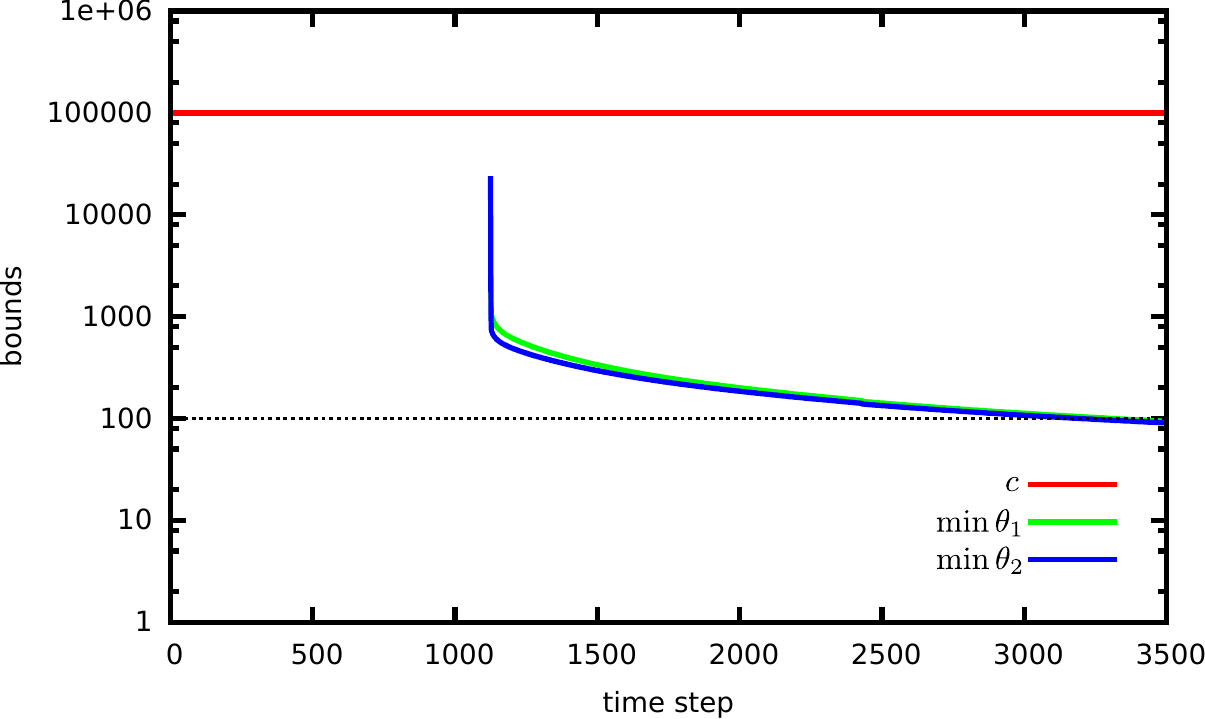}
\end{center}
 \caption{The three terms of the time step restriction, plotted as functions over the time step number; dotted line: used time step size $\tau=100\,s$.}
 \label{fig:step-size-restriction}
\end{figure}

In a second step, we try to numerically assess the time step bound~\eqref{estt}.
The proof of Theorem~\ref{existence_reg}
shows that a nonnegative surface water height $w^n$ at time step $n$ implies
a nonnegative surface water height $w^{n+1}$ at time step $t_{n+1}$ if the time step
size $\tau = t_{n+1} - t_n$ is bounded by \eqref{estt}, which states
\begin{equation}
\label{eq:numerics_time_step_size_restriction}
\tau
 \leq
\min \{ c, \theta_1(x), \theta_2(x) \}
\qquad
\text{for almost all $x \in \Sigma_\text{in}$},
\end{equation}
where 
\begin{align*}
 \theta_1(x)
 & =
 \begin{cases}
  \frac{c \sigma}{\sigma - cr(x,t_n) + h(u^n(x))}
  & \text{if this term is nonnegative}, \\
   \infty & \text{else}
 \end{cases} \\
\intertext{and}
 \theta_2(x)
 & =
 \begin{cases}
  \frac{c \sigma}{\sigma + h(u^n(x))} & \text{if this term is nonnegative}, \\
   \infty & \text{else.}
 \end{cases}
\end{align*}
Note that by construction, the minimum in \eqref{eq:numerics_time_step_size_restriction}
is always a finite nonnegative number.  

In Fig.~\ref{fig:step-size-restriction} we have plotted $c$,
$\min_{x \in \Sigma_\text{in}} \theta_1(x)$ and $\min_{x \in \Sigma_\text{in}} \theta_2(x)$, with $u^n$ replaced by $u^n_j$ for $j=4$, 
as functions of the time step number~$n$.  We observe that our time step is below the
bound~\eqref{estt} for almost all time steps.  In accordance with the continuous theory
we observe nonnegativity of the solution everywhere, with one notable exception.  A surface water element
with initial value $w_p =0$, but next to an element $q$ with $w_q > 0$ will drop below zero
if the subsurface pressure is negative there.  This is a discretization effect.
Unlike in the continuous case, the surface water levels $w_p$ and $w_q$ couple
through a subsurface basis function.  Since $w_q > 0$ and the subsurface pressure is negative,
there will be a flux into the subsurface which will in turn suck water out of the element $p$,
making $w_p$ negative.  Through this effect, negative surface water heights down to $-0.013$\,m
where produced in our example.

It is plausible to try to use the bound~\eqref{eq:numerics_time_step_size_restriction}, with $u^n$ replaced by $u^n_j$,  as
a time step control mechanism.
To check whether this bound is sharp, we have recomputed our example with the time step sizes 
$\tau = 50\,\text{s}$, $100\,\text{s}$, $200\,\text{s}$, $400\,\text{s}$, $800\,\text{s}$,
$1\,600\,\text{s}$, $3\,200\,\text{s}$.  Only at the last value of $\tau$ we do observe
instabilities presumably caused by violations of the CFL condition.  As it turns out,
the graphs in Fig.~\ref{fig:step-size-restriction} are virtually independent of the
time step size.  Also, we do not see negative water levels (except for the discretization
effect described above) for any of the time step sizes considered.
Hence we conclude that the sufficient condition
\eqref{eq:numerics_time_step_size_restriction} in the spatially continuous case is generally to tight when applied to fully discretized problems.

\subsection*{Acknowledgments} We thank Ben Schweizer for fruitful discussions during this work. 

\bibliographystyle{plainnat}

\end{document}